\documentclass[pdflatex,sn-mathphys]{sn-jnl}

\usepackage[english]{babel}
\usepackage{amsfonts}
\usepackage{graphicx}
\usepackage{multirow}
\usepackage{multicol}
\usepackage{xcolor}
\usepackage[short]{optidef}
\usepackage{changepage}
\usepackage{lmodern}
\usepackage{hyperref}
\usepackage{comment}
\usepackage{enumitem}
\usepackage{booktabs}
\usepackage{tabularx}
\usepackage{ltablex}
\usepackage{longtable}
\usepackage{comment}
\usepackage{url}
\usepackage{algorithm}

\DeclareMathOperator*{\argmin}{arg\,min}
\DeclareMathOperator*{\argmax}{arg\,max}

\newcommand{\Diag}{\textrm{Diag}}
\newcommand{\diag}{\textrm{diag}}
\newcommand{\tr}{\textrm{tr}}
\newcommand{\ip}[2]{\left\langle #1,#2\right\rangle}
\newcommand{\inprod}[2]{\langle #1,#2\rangle}

\jyear{2021}%

\theoremstyle{thmstyleone}%
\newtheorem{theorem}{Theorem}
\newtheorem{lemma}{Lemma}
\newtheorem{proposition}[theorem]{Proposition}%

\theoremstyle{thmstyletwo}%
\newtheorem{remark}{Remark}%

\theoremstyle{thmstylethree}%

\begin{document}

\title[Exact Cardinality-constrained Clustering]{Global Optimization for Cardinality-constrained Minimum Sum-of-Squares
Clustering via Semidefinite Programming}


\author*[1]{\fnm{Veronica} \sur{Piccialli}}\email{veronica.piccialli@uniroma1.it}

\author[1]{\fnm{Antonio M.} \sur{Sudoso}}\email{antoniomaria.sudoso@uniroma1.it}



\affil[1]{\orgdiv{Department of Computer, Control and Management Engineering}, \orgname{Sapienza University of Rome}, \orgaddress{\street{Via Ariosto 25}, \city{Rome}, \postcode{00185}, \country{Italy}}}




\abstract{The minimum sum-of-squares clustering (MSSC), or k-means type clustering, has been recently extended to exploit prior knowledge on the cardinality of each cluster. Such knowledge is used to increase performance as well as solution quality. In this paper, we propose a global optimization approach based on the branch-and-cut technique to solve the cardinality-constrained MSSC. For the lower bound routine, we use the semidefinite programming (SDP) relaxation recently proposed by Rujeerapaiboon et al. [SIAM J. Optim. 29(2), 1211-1239, (2019)]. However, this relaxation can be used in a branch-and-cut method only for small-size instances. Therefore, we derive a new SDP relaxation that scales better with the instance size and the number of clusters. In both cases, we strengthen the bound by adding polyhedral cuts. Benefiting from a tailored branching strategy which enforces pairwise constraints, we reduce the complexity of the problems arising in the children nodes. For the upper bound, instead, we present a local search procedure that exploits the solution of the SDP relaxation solved at each node. Computational results show that the proposed algorithm globally solves, for the first time, real-world instances of size 10 times larger than those solved by state-of-the-art exact methods.}




\maketitle

\section{Introduction}
Cluster analysis or clustering is the task of partitioning similar objects into different groups according to some defined distance measure \cite{rao1971cluster}.
Clustering is a fundamental technique for data analysis that belongs to a subclass of unsupervised learning algorithms in machine learning and statistics. Its relevance is supported by a large number of algorithms and applications such as customer segmentation, medical imaging, recommendation systems, social network analysis, and image processing \cite{ganclustering}. The complexity of finding a suitable clustering significantly depends on the fitness measure of a proposed partition. Among many criteria used in cluster analysis, the most studied and frequently adopted criterion is the minimum sum-of-squares clustering (MSSC). Given a set of $n$ data points $\{p_i\}_{i=1}^n$ in the $d$-dimensional Euclidean space, the MSSC problem aims to partition them into $k$ clusters $\{C_j\}_{j=1}^k$, so that the total sum of squared Euclidean distances between the data points and the cluster centers $\{m_j\}_{j=1}^k$ is minimized. 
The MSSC problem is known to be NP-hard \cite{aloise2009complexity}, since it leads to a non-convex mixed-integer nonlinear optimization problem (MINLP) that is challenging to solve in practice \cite{rujeerapaiboon2019size}. This is why the MSSC is frequently addressed by means of heuristic algorithms in comparison with exact methods. 

In recent years, researchers have started to focus on clustering with user constraints, i.e., instance-level and cluster-level constraints, to make the clustering process more accurate \cite{davidson2007survey}. Instance-level constraints, typically must-link and cannot-link constraints, indicate that two points must or cannot be assigned to the same cluster \cite{wagstaff2001constrained, baumann2020binary}. Cluster-level constraints, instead, exploit prior knowledge on the structure of the clusters by imposing fixed cluster sizes or lower and upper bounds on the capacity of each cluster \cite{banerjee2006scalable, zhu2010data, gnagi2021matheuristic}. Cardinality constraints find relevant applications in demand planning and segmentation \cite{mancuso2021machine}, electric power systems research \cite{balletti2022mixed}, document clustering \cite{hu2008towards}.
{Both in unconstrained and constrained clustering, the presence of outliers can deteriorate the quality of the solution and should be carefully considered. Strict cardinality constraints offer a framework to keep into account the presence of outliers, improving the quality of the solution of the clustering problem \cite{rujeerapaiboon2019size}. This is achieved in \cite{rujeerapaiboon2019size} by allocating an extra cluster to accomodate outliers.
}
For a detailed overview of constrained clustering applications see the survey in \cite{ganccarski2020constrained} and references therein. 
In this paper, we focus on the MSSC problem with strict cardinality constraints, in short ccMSSC. 
Adding constraints to an existing clustering formulation may make the resulting problem harder, both empirically and in terms of worst-case analysis  \cite{liberti2021side}. 
However, if clustering problems are only solved by means of heuristic algorithms, unsatisfactory results may be obtained \cite{randel2021lagrangian}. 
For this reason, global algorithms play a fundamental role in two different aspects. First, a certified globally optimal solution is fundamental for evaluating, improving, and developing heuristics and approximation algorithms, and second, as an unsupervised machine learning task, results usually require interpretation from domain experts. Since the majority of algorithms can locate only local minimizers, this interpretation may be completely erroneous in case the obtained clustering solution is far from the global optimum.



Our motivation for studying ccMSSC is twofold. First, constrained clustering with instance-level constraints is a well-covered topic in the literature with both heuristic and global optimization algorithms. Clustering with cardinality constraints, instead, is less studied and the existing methods are only capable of solving instances with a very limited number of data points, i.e., less than $n=150$. Second,
clustering with cardinality constraints is robust with respect to unfair solutions containing very few or even no data points. 
In this paper, we propose a global optimization algorithm for constrained MSSC where the prior knowledge is incorporated in the form of strict cardinality constraints. This method is based on the branch-and-cut technique and exploits semidefinite programming (SDP) tools to obtain tight lower and upper bounds.
The main original contributions of this paper are:
\begin{enumerate}
\item We derive a new SDP relaxation for ccMSSC and we compare it theoretically and numerically with the SDP relaxation recently proposed in \cite{rujeerapaiboon2019size}. The new bound can be weaker than the one in \cite{rujeerapaiboon2019size}, but is solved significantly faster as $n$ and $k$ increase.
\item We strengthen the new SDP bound and the one in \cite{rujeerapaiboon2019size} by adding polyhedral cuts and we attack the resulting SDP relaxations by means of a cutting-plane algorithm. 
\item We design a deterministic rounding procedure providing an upper bound on the optimum value of ccMSSC. We use a variant of the $k$-means algorithm in conjunction with the solution of the SDP relaxation to initialize the cluster centers.
\item We propose an SDP-based branch-and-cut algorithm producing certifiably optimal solutions for ccMSSC. We use both SDP relaxations for computing the lower bound, exploiting the strength of the one in \cite{rujeerapaiboon2019size} for small instances and the efficiency of the new one for larger instances.
\item Benefiting from a tailored branching strategy, we manage to reduce the size of the problem in the children nodes by solving equivalent SDPs over lower dimensional positive semidefinite cones.
\item Our exact algorithm solves for the first time to global optimality real-world instances of size 10 times larger than those solved by exact algorithms proposed in the literature. 
\end{enumerate}
The rest of this paper is organized as follows. Section \ref{sec:review} reviews the literature related to the ccMSSC problem. In Section \ref{sec:sdp_relax}, the existing SDP relaxation of ccMSSC and the new SDP bound are described. In Section \ref{sec:sdp_safe}, problem-specific post-processing techniques yielding valid lower bounds are described. In Section \ref{sec:cutting_plane}, the cutting-plane algorithm used for computing the bound is proposed. A problem-specific branching rule is shown in Section \ref{sec:branching}, and a primal heuristic is introduced in Section \ref{sec:heuristic}. Computational results are reported in Section \ref{sec:results} and some directions for future work conclude the paper in Section \ref{sec:conclusions}.

\section{Related Work}
\label{sec:review}
To the best of our knowledge, the first heuristic algorithm for ccMSSC can be found in \cite{bradley2000constrained}. They use the classical Lloyd's algorithm \cite{lloyd1982least} in conjunction with the solution of a linear assignment problem requiring that each cluster contains at least a minimum number of points. Since then, not many heuristics have been proposed in the literature. In \cite{banerjee2006scalable, malinen2014balanced, costa2017less}, different heuristic algorithms are proposed to solve the balanced version of ccMSSC where the clusters have the same number of data points, i.e., $\lvert C_1 \rvert = \lvert C_2 \rvert = \ldots = \lvert C_k \rvert = n/k$.
Although these heuristics tend to produce feasible solutions in a reasonable amount of time, their quality highly depends on the choice of the initial solution and, more importantly, they do not provide optimality guarantees. In the literature, several global optimization algorithms have been proposed for both unconstrained and constrained MSSC problems. In this direction, there are two main methodological approaches: (a) mathematical programming techniques based on branch-and-bound (B\&B) or column generation and (b) constraint programming (CP) techniques. The first class of methods is predominant in the unconstrained literature \cite{aloise2009branch,
aloise2012improved, krislock2016computational, piccialli2022sos} with some extensions designed to solve the MSSC with instance-level constraints \cite{xia2009global, babaki2014constrained, piccialli2022exact}. However, most of the recent contributions to constrained variants come from the CP community. Among the exact methods for constrained MSSC, the method presented in \cite{duong2013declarative} is the first attempt at using CP for MSSC. Other successful approaches based on the CP paradigm can be found in \cite{duong2017constrained, guns2016repetitive}.
Turning now to the MSSC with cardinality constraints, all the CP approaches proposed for the MSSC with instance-level constraints can be extended to solve the ccMSSC by simply adding a global cardinality constraint over a set of variables. However, as described in \cite{haouas2020exact}, a tailored approach consisting of specialized global constraints coupled with filtering algorithms and search heuristics can quickly reduce the search space and achieve better performance. Despite the improvement with respect to previous CP approaches, the algorithm in \cite{haouas2020exact} exhibits high computational times on small instances, and for some of them, it does not succeed in certifying the optimality of the produced clustering. To the best of our knowledge, this CP approach is the current state-of-the-art global optimization algorithm of ccMSSC.

In recent years, several conic optimization approaches have been proposed in the literature. In particular, due to their effectiveness and strong theoretical properties, there is a large branch of literature toward the application of techniques from semidefinite programming (SDP) \cite{vandenberghe1996semidefinite}. Peng and Wei \cite{peng2007approximating} prove the equivalence between the MSSC problem and a nonlinear SDP reformulation, the so-called 0-1 SDP, and they provide an SDP relaxation. Since then, the Peng-Wei model has been studied from the theoretical point of view developing recovery guarantees and conditions under which the relaxation recovers the underlying clusters with high probability \cite{awasthi2015relax, iguchi2017probably, li2020birds, de2022ratio}. Recently, it has been successfully embedded in B\&B algorithms to globally solve unconstrained MSSC \cite{piccialli2022sos} and its variant with pairwise constraints \cite{piccialli2022exact}. 
To the best of our knowledge, the only conic optimization scheme for ccMSSC has been proposed in \cite{rujeerapaiboon2019size}. Although the focus is not on exact methods, the authors propose lower bounds based on semidefinite and linear programming relaxations and a rounding heuristic to find a feasible clustering satisfying the constraints. 

{The LP and SDP relaxations in \cite{rujeerapaiboon2019size} are proven to recover the globally optimal solution whenever a cluster separation condition is met, but only in the balanced case. The separation condition requires all cluster diameters to be smaller than the distance between any two distinct clusters. In other words, for datasets whose hidden classes are balanced and well separated, they succeed in recovering the provably optimal clustering.
However, when this condition is not met, they are unable to prove optimality since the upper bound produced by the rounding procedure does not coincide with the lower bound obtained via the relaxations. Furthermore, in real-world applications, data may not be generated according to a distribution that yields well-separated clusters, and this may result in unbalanced and overlapping clusters, implying that the theory behind cluster recovery does not apply. For this reason, in this paper, we 
certify the globally optimal solution by using global optimization tools that do not require assumptions on the data distribution and/or on the clusters' structure.} In this paper we build on the work in \cite{rujeerapaiboon2019size} and we design the first SDP-based branch-and-cut algorithm for solving ccMSSC. 



\subsection*{Notation}
Throughout this paper, $[n]$ denotes the set $\{1, \dots, n\}$ of indices of the data points, and $[k]$ denotes the set $\{1, \dots, k\}$ of indices of the clusters. Let $\mathcal{S}^n$ be the set of all $n\times n$ real symmetric matrices. We denote by $M\succeq 0$ a positive semidefinite matrix $M$ and by $\mathcal{S}_+^n$ be the set of all positive semidefinite matrices of size $n\times n$. Analogously, we denote by $M \succ 0$ a positive definite matrix $M$ and let ${\mathcal S}_{++}^n$ be the set of all positive definite matrices of size $n\times n$. We denote by $\inprod{\cdot}{\cdot}$ the
trace inner product. That is, for any
$A, B \in \mathbb{R}^{n\times n}$, we define $\inprod{A}{B}:= \tr (B^\top A)$. Given a matrix $A$, we denote by $A_{i, :}$ and $A_{:,j}$ the $i$-th row and the $j$-th column of $A$, respectively. Finally, we denote by $1_n$ the vector of all ones of size $n$ and by $I_n$ the identity matrix of size $n \times n$.

\section{SDP Relaxations for ccMSSC}\label{sec:sdp_relax}
We study the ccMSSC problem, which is defined as
the task of partitioning $n$ data points $p_1, \dots, p_n$, where $p_i \in \mathbb{R}^d$,
into $k$ clusters of \textit{known} sizes
$c_1, \dots, c_k$, with $\sum_{j=1}^k c_j = n$ and $c_j \in \mathbb{Z}_+$, such that the total sum-of-squared intra-cluster distances is minimized. It can be formulated as the following discrete optimization problem:
\begin{subequations}
\begin{alignat}{2}\tag{ccMSSC}
&\!\min        &\quad& \frac{1}{2} \sum_{j=1}^{k} \frac{1}{c_j} \sum_{s=1}^{n} \sum_{t=1}^{n} d_{st} (\pi_j)_s (\pi_j)_t\label{prob:ccMSSC1} \\
&\text{s.t.} &      & \sum_{j=1}^{k} (\pi_j)_i = 1 \ \quad  \forall i \in [n],\label{ccMSSC:hard}\\
&                  &      & \sum_{i=1}^{n} (\pi_j)_i = c_j \quad  \forall j \in [k],\label{ccMSSC:cardinality}\\
&                  &      & \pi_j \in \{0, 1\}^n \ \ \quad \forall j \in [k].
\end{alignat}
\end{subequations}
Here, $d_{st}$ is the squared Euclidean distance between $p_s$ and $p_t$ computed as $\|p_s - p_t\|_2^2$ and $\pi_j = [(\pi_j)_1, \dots, (\pi_j)_n]^\top$ is the indicator variable of cluster $j$, i.e., the $i$-th component of $\pi_j$ is set to 1 is the $i$-th data point is assigned to cluster $j$ and 0 otherwise. Constraints (\ref{ccMSSC:hard}) ensure that each data point is assigned to exactly one cluster and constraints (\ref{ccMSSC:cardinality}) ensure that the cluster $j$ contains exactly $c_j$ data points.
For general $n$ and $k$, it is known that the MSSC problem with cardinality constraints is NP-hard \cite{rujeerapaiboon2019size}.
In this paper we are interested in exact solutions that we generate by means of a branch-and-cut algorithm. Therefore, we need efficiently computable lower and upper bounds on the original problem. For the ccMSSC problem, lower bounds obtained by solving the LP relaxation are weak in practice \cite{rujeerapaiboon2019size}. On the contrary, the bound provided by SDP relaxations strengthened through valid inequalities we will introduce next is much stronger. This motivates us to study SDPs that can be efficiently used within a branch-and-cut framework.
\subsection{Vector lifting SDP relaxation}
\noindent
In this section we review the relaxation technique presented in \cite{rujeerapaiboon2019size} for problem \eqref{prob:ccMSSC1}. 
{We present the formulation introduced in \cite{rujeerapaiboon2019size}, where we shift the decision variables from $\{-1,1\}$ to $\{0,1\}$. This change of variables allows us to better relate the relaxation in \cite{rujeerapaiboon2019size} to the new one introduced in this paper}. Let $\Pi_j$ be the $n\times n$ symmetric matrix given by $\Pi_j = \pi_j (\pi_j)^\top$ for $\pi_j \in \{0, 1\}^n$ and $j \in [k]$. This implies $\Pi_j \succeq \pi_j (\pi_j)^\top$, $\diag(\Pi_j) = \pi_j$ and $\textrm{rank}(\Pi_j) = 1$. Conversely, for any matrix $\Pi_j$ satisfying $\Pi_j \succeq 0$, $\diag(\Pi_j) = \pi_j$ and $\textrm{rank}(\Pi_j) = 1$ we have $\Pi_j = \pi_j (\pi_j)^\top$ for some $\pi_j \in \{0, 1\}^n$ \cite{vandenberghe1996semidefinite}. 
Denote by $D$ the {Euclidean distance matrix}, i.e., $D_{ij} = \| p_i - p_j\|_2^2$. Problem \eqref{prob:ccMSSC1} can be exactly rewritten as
\begin{mini}[2]
{}{\frac{1}{2} \left \langle D, \sum_{j=1}^{k} \frac{1}{c_j} \Pi_j \right \rangle}
    {\label{prob:vlSDPnonconvex}}{}
\addConstraint{\sum_{j=1}^{k} \pi_j}{= 1_n}{}
\addConstraint{1_n^\top \pi_j}{= c_j}{\quad \forall j \in [k]}
\addConstraint{\Pi_j}{\succeq \pi_j (\pi_{j})^\top}{\quad  \forall j \in [k]}
\addConstraint{\textrm{rank}(\Pi_j)}{= 1}{\quad  \forall j \in [k]}
\addConstraint{\diag(\Pi_j)}{= \pi_j}{\quad \forall j \in [k]}.
\end{mini}
Because of the rank-one constraints, this is not a semidefinite program. 
Constraint $\Pi_j - \pi_j (\pi_j)^\top \succeq 0$ can be reformulated, by the Schur's complement, as the lifted constraints
\begin{equation*}
{Y_j=}   \begin{bmatrix}
1 & (\pi_j)^\top \\
\pi_j & \Pi_j
\end{bmatrix} \in \mathcal{S}_{+}^{n+1} \qquad \forall j \in [k].
\end{equation*}
Furthermore, looking at the structure of $\Pi_j$ additional constraints can be added to strengthen the formulation. It easily follows that $\Pi_j 1_n = c_j \pi_j$ and $\Pi_j \geq 0_{n \times n}$ for all $j \in [k]$.

{Let $W \in \mathcal{S}^n_+$ be the matrix of the inner products of the data points, i.e., $W_{ij} = p_i^\top p_j$. We now rewrite the objective function using the {Lindenstrauss mapping} between the {Euclidean distance matrix} $D$ and the Gram matrix $W$ on the clustering feasible set \cite{edmBook, alfakih2018euclidean}:
\begin{align*}
    D = \diag(W)1_n^\top + 1_n \diag(W)^\top - 2 W.
\end{align*}
The objective function of problem \eqref{prob:vlSDPnonconvex} can be rewritten as
{\footnotesize
\begin{align*}
    \frac{1}{2} \left\langle D, \sum_{j=1}^{k} \frac{1}{c_j} \Pi_j \right\rangle &= \frac{1}{2} \left\langle \diag(W)1_n^\top, \sum_{j=1}^{k} \frac{1}{c_j} \Pi_j \right\rangle + \frac{1}{2} \left\langle 1_n \diag(W)^\top, \sum_{j=1}^{k} \frac{1}{c_j} \Pi_j \right\rangle - \left\langle W, \sum_{j=1}^{k} \frac{1}{c_j} \Pi_j \right\rangle \\
    &= \frac{1}{2} \ip{\diag(W)}{\sum_{j=1}^{k} \frac{1}{c_j} \Pi_j 1_n} + \frac{1}{2} \ip{\diag(W)}{\sum_{j=1}^{k} \frac{1}{c_j} (\Pi_j)^\top 1_n} -  \left\langle W, \sum_{j=1}^{k} \frac{1}{c_j} \Pi_j \right\rangle \\
    & = \frac{1}{2} \diag(W)^\top 1_n + \frac{1}{2} 1_n^\top \diag(W) - \ip{W}{\sum_{j=1}^{k} \frac{1}{c_j} \Pi_j} = \tr(W) - \ip{W}{\sum_{j=1}^{k} \frac{1}{c_j} \Pi_j},
\end{align*}}
where the last equality derives from the constraints linking $\Pi_j$ and $\pi_j$, i.e., $\Pi_j 1_n = c_j \pi_j$ and $\sum_{j=1}^{k} \pi_j = 1_n$.}
Therefore, the resulting SDP relaxation is the convex program
\begin{equation}
\begin{aligned}\label{prob:vlSDP}
f^\star_{\textrm{VL}} = \min_{} \quad & {\ip{W}{I_n-\sum_{j=1}^{k} \frac{1}{c_j} \Pi_j}}\\
\textrm{s.t.} \quad & \sum_{j=1}^{k} \pi_j= 1_n\\
& (\pi_j, \Pi_j) \in \mathcal{S}_{\textrm{VL}}(c_j) \quad \forall j \in [k]   \\
\end{aligned}\tag{VL-SDP}
\end{equation}
where, for any $c \in \mathbb{Z}_+$, the set $\mathcal{S}_{\textrm{VL}}(c) \subset \mathbb{R}^n \times \mathcal{S}^n$ is defined as
\begin{align*}
    \mathcal{S_{\textrm{VL}}}(c) = \left\{(\pi, \Pi) \in \mathbb{R}^n \times \mathcal{S}^n : \begin{array}{l}
    1_n^\top \pi = c, \ \diag(\Pi) = \pi, \ \Pi 1_n = c \pi, \\
    \Pi \succeq \pi (\pi)^\top, \ \Pi \geq 0_{n \times n} \\
    \end{array} \right\}.
\end{align*}
From now on, we refer to the relaxed problem as ``vector lifting'' SDP relaxation. This kind of SDP is a doubly nonnegative program (DNN) since the matrix variable is both positive semidefinite and elementwise nonnegative. Problem \eqref{prob:vlSDP} provides a lower bound on the optimal objective value of problem \eqref{prob:ccMSSC1}. Moreover, if the
optimal solution $\{(\pi_j^\star, \Pi_j^\star)\}_{j=1}^k$ of problem \eqref{prob:vlSDP} satisfies $\Pi_j = \pi_j (\pi_j)^\top$ for all $j \in [k]$, then we can conclude that  $\{(\pi_j^\star, \Pi_j^\star)\}_{j=1}^k$ is an optimal solution of problem \eqref{prob:ccMSSC1}.
Note that, this SDP relaxation
has $k$ matrix variables of size $(n+1) \times (n+1)$, thus requires $O(kn^2)$ variables and constraints. 
Although the number of clusters $k$ is much smaller than $n$, the relaxation becomes impractical for branch-and-bound algorithms as $n$ and $k$ grow. In the following, we substantially reduce the number of variables by introducing a new SDP bound. This relaxation is based on the ``matrix lifting'' technique \cite{mittelmann2010estimating, ding2011equivalence} and only has one matrix variable of size $(n+k) \times (n+k)$. We will show that this reduction in the number of variables leads to a significant computational efficiency compared to the vector lifting SDP relaxation.
{Note that, the matrix lifting technique has been also successfully employed in graph partitioning \cite{wolkowicz1999semidefinite, li2021strictly}, where the nodes of a graph have to be partitioned into clusters in such a way to minimize the weights of the edges among different clusters.}


\subsection{Matrix lifting SDP Relaxation}

\noindent
We now describe the new bound obtained from our matrix lifting SDP relaxation. This relaxation uses only {$O\left((n+k)^2\right)$} variables and constraints. 
Since the computational complexity of solving SDPs is a polynomial function of the number of variables, we can expect a significant complexity reduction with respect to the vector lifting SDP relaxation. In the following, we briefly review the well-known discrete optimization model for unconstrained MSSC and the SDP relaxation proposed in \cite{peng2007approximating}. Let $X$ be the $n \times k$ assignment matrix, i.e., $X_{ij} = 1$ if the $i$-th data point is assigned to the $j$-th cluster and 0 otherwise. The Peng-Wei MSSC discrete model in matrix notation is
\begin{equation}
\begin{aligned}\label{prob:PWdiscrete}
\min_{} \quad & \tr(W - W X (X^\top X)^{-1} X^\top)\\
\textrm{s.t.} \quad & X 1_k = 1_n, \ X^\top 1_n \geq 0_k, \ X \in \{0, 1\}^{n \times k}.
\end{aligned}
\end{equation}
and setting $Z = X(X^\top X)^{-1}X^\top$, the following SDP relaxation can be derived
\begin{equation}
\begin{aligned}\label{prob:SDPpeng}
f^\star_\textrm{PW}=\min_{} \quad & \tr(W - W Z)\\
\textrm{s.t.} \quad & Z 1_n=1_n, \ \tr(Z)=k, \ Z \in \mathcal{S}^n_+, \ Z \geq 0_{n \times n}.
\end{aligned}
\end{equation}
Note that constraints $Z \geq 0_{n \times n}$ and $Z 1_n = 1_n$ ensure that $Z$ is a stochastic matrix, and hence all of its eigenvalues lie between 0 and 1. We now extend the above setting to incorporate cardinality constraints. Let $C = \Diag(c_1, \dots, c_k)$, be the diagonal matrix containing the cluster sizes. 
An exact reformulation of ccMSSC in matrix notation is
\begin{equation}
\begin{aligned}\label{prob:ccMSSC2}
\min_{} \quad & \tr(W - W X C^{-1} X^\top)\\
\textrm{s.t.} \quad & X 1_k= 1_n, \ X^\top 1_n = \diag(C), \ X \in \{0, 1\}^{n \times k}\\
\end{aligned}
\end{equation}
where $X$ is the assignment matrix. Clearly, problem \eqref{prob:ccMSSC1} is equivalent to problem \eqref{prob:ccMSSC2}. To see this, note that the $j$-th column of $X$ is the characteristic vector of the cluster $j$, i.e, $\pi_j = X_{:,j}$ for all $j \in [k]$, and the $k$ clusters are in one-to-one correspondence with the set of partition matrices
\begin{align}\label{set:ass}
    \mathcal{F} = \left\{ X \in \{0, 1\}^{n \times k} : \ X 1_k = 1_n, \ X^\top 1_n = \diag(C) \right\}.
\end{align}
To obtain the matrix lifting SDP relaxation of problem \eqref{prob:ccMSSC2}, we first linearize the objective function by introducing a new matrix variable $Z = X C^{-1} X^\top$. 
This yields the feasible set
\begin{equation}
    \mathcal{\bar{F}} = \textrm{conv}\left\{ Z \in \mathcal{S}^{n+k} : \ Z = XC^{-1}X^\top, \ X \in \mathcal{F} \right\}.
\end{equation}
Therefore, we can rewrite the problem as
\begin{equation}
\begin{aligned}
\min_{Z \in \mathcal{\bar{F}}} \quad & \textrm{tr}(W - W Z)\\
\end{aligned}
\end{equation}
In order to approximate the set $\mathcal{\bar{F}}$, we relax the integrality constraint on $X$ and we set $X \geq 0_{n \times k}$. Next, we replace the non-convex constraint $Z = X C^{-1} X^\top$ by $Z \succeq X C^{-1} X^\top$ and by using the Schur's complement we obtain the equivalent linear matrix inequality
\begin{equation}
{Y} = \begin{bmatrix}
C & X^\top \\
X & Z
\end{bmatrix} \in \mathcal{S}^{n+k}_+.
\end{equation}
Although constraint $Z = X C^{-1} X^\top$ is relaxed, we can still consider some additional linear constraints to further improve the quality of the solution. 
Since $Z$ approximates $X C^{-1} X^\top$ and $X \geq 0_{n \times k}$, we have $Z \geq 0_{n \times n}$, $\diag(Z) = X \diag(C^{-1})$ and
\begin{align}
    Z 1_n &= X C^{-1} X^\top 1_n = X C^{-1} \textrm{diag}(C) = X 1_k = 1_n.
\end{align}
By collecting all the mentioned constraints, the matrix lifting SDP relaxation is given by
\begin{equation}
\begin{aligned}\label{prob:mlSDP}
f^\star_\textrm{ML} = \min_{(X, Z) \in \mathcal{S}_{\textrm{ML}}} \quad & \tr(W - W Z)
\end{aligned}\tag{ML-SDP}
\end{equation}

where the convex set $\mathcal{S}_{\textrm{ML}} \subset \mathbb{R}^{n \times k} \times \mathcal{S}^n$ is defined as

{\small
\begin{align}
    \mathcal{S_{\textrm{ML}}} = \left\{(X, Z) \in \mathbb{R}^{n \times k} \times \mathcal{S}^n : \begin{array}{l}
    X 1_k = 1_n, \ X^\top 1_n = \diag(C), \\
    Z 1_n = 1_n, \ \diag(Z) = X \diag(C^{-1}), \\
    X \geq 0_{n \times k}, \ Z \geq 0_{n \times n}, \ Z \succeq X C^{-1} X^\top \\
    \end{array} \right\}.
\end{align}}

Similarly to \eqref{prob:vlSDP}, problem \eqref{prob:mlSDP} is a DNN program whose optimal value yields a lower bound on the optimal value of problem \eqref{prob:ccMSSC1}. Moreover, if the
optimal solution $(X^\star, Z^\star)$ is an extreme point of $\mathcal{\bar{F}}$, then we can conclude that $X^\star_{:,j} = \pi^\star_j$ for all $j \in [k]$ is an optimal solution of problem \eqref{prob:ccMSSC1}.

It is interesting to compare problem \eqref{prob:mlSDP} with the Peng-Wei SDP relaxation of unconstrained MSSC in \eqref{prob:SDPpeng}. The matrix $Z$ in problem \eqref{prob:mlSDP} shares the same structure of the one in \eqref{prob:SDPpeng}. This implies that valid inequalities used in \cite{aloise2009branch} and \cite{piccialli2022sos} for strengthening the Peng-Wei bound can also be exploited to tighten our relaxation. 
Finally, we point out that, in contrast to our relaxation, the Peng-Wei SDP does not have access to either the cluster size or the assignment matrix. We will discuss how to exploit the SDP solution to recover feasible clustering solutions in Section \ref{sec:heuristic}.

One special case of interest is the balanced ccMSSC where each cluster contains the same number of data points, i.e., $c_j = \frac{n}{k}$ for all $j \in [k]$. This leads to the Amini-Levina SDP relaxation, which is presented in \cite{amini2018semidefinite} and takes the form
\begin{equation}
\begin{aligned}\label{prob:SDPamini}
f^\star_\textrm{AL} = \min_{} \quad & \tr(W - W Z)\\
\textrm{s.t.} \quad & Z 1_n=1_n, \ \diag(Z)=\frac{k}{n}1_n, \ Z \in \mathcal{S}^n_+, \ Z \geq 0_{n \times n}.
\end{aligned}
\end{equation}
Our \eqref{prob:mlSDP} generalizes the Amini-Levina SDP to the case of unbalanced cluster sizes and it can be shown to be equivalent to problem \eqref{prob:SDPamini} when the cardinalities of all clusters are the same. In order to prove this result, we need the following lemma.
\begin{proposition}
Let $C = \Diag(\frac{n}{k}, \dots, \frac{n}{k})$, then Problems \eqref{prob:mlSDP} and \eqref{prob:SDPamini} are equivalent.
\end{proposition}
\begin{proof}
We show that any feasible solution of problem \eqref{prob:mlSDP} gives rise to a feasible solution of problem \eqref{prob:SDPamini} with the same objective value and vice versa. Let $(X, Z)$ be feasible for \eqref{prob:mlSDP} then it immediately follows that $\bar{Z} = Z$ is feasible for \eqref{prob:SDPamini} and achieves the same objective value. Conversely, let $Z$ be feasible for problem \eqref{prob:SDPamini} then we construct a solution $(\bar{X}, \bar{Z})$ of problem \eqref{prob:mlSDP} with $\bar{Z} = Z$ and $\bar{X} = \frac{1}{k} 1_n 1_k^\top$. Then, we have $\bar{X} 1_k = 1_n$, $\bar{X}^\top 1_n = \diag(C)$ and $\diag(\bar{Z}) = \bar{X} \diag(C^{-1}) = \bar{X} \frac{k}{n} 1_k = \frac{k}{n} 1_n$. Let $M = \bar{Z} - \bar{X}C^{-1}\bar{X}^\top = \bar{Z} - \frac{1}{n} 1_n 1_n^\top$. It remains to show that $ M \succeq 0$. From $\bar{Z}1_n = 1_n$ we have that $1$ is an eigenvalue of $\bar{Z}$ with corresponding eigenvector $1_n$.
The positive semidefiniteness of $M$ follows from $\bar Z$ being positive semidefinite and $M=\bar{Z} - \frac{1}{n} 1_n 1_n^\top$ being a deflation matrix of $\bar Z\succeq 0$ with respect to the eigenvalue-normalized eigenvector pair $\left(1,\frac{1}{\sqrt{n}}1_n\right)$.
\end{proof}

We now prove that our relaxation \eqref{prob:mlSDP} is dominated by \eqref{prob:vlSDP}, the best-known relaxation of the ccMSSC problem. Then, we compare the SDPs numerically on some instances from the literature. We show that the bound provided by problem \eqref{prob:mlSDP} is competitive and the computational effort to compute it is much smaller than the one required to solve problem \eqref{prob:vlSDP}.

\begin{proposition}
We have $f^\star_{\textrm{VL}} \geq f^\star_{\textrm{ML}} \geq f^\star_{\textrm{PW}}$.
\end{proposition}
\begin{proof}
We first show $f^\star_{\textrm{VL}} \geq f^\star_{\textrm{ML}}$. For any feasible solution $\{(\pi_j, \Pi_j)\}_{j=1}^k$ of problem \eqref{prob:vlSDP}, we can construct a solution $(Z, X)$ with
\begin{align}
    Z = \sum_{j=1}^{k} \frac{1}{c_j} \Pi_j \quad \textrm{and} \quad X = [\pi_1, \dots, \pi_k],
\end{align}
that is feasible for problem \eqref{prob:mlSDP} and achieves the same objective value.
By construction we have $X \geq 0_{n \times n}$ and $Z \geq 0_{n \times n}$ since $\pi_j$ is nonnegative due to $\Pi_j \geq 0_{n \times n}$ and  $\diag(\Pi_j) = \pi_j$ for $j=1,\dots,k$. Similarly, $X 1_k = 1_n$ and $ X^\top 1_n = \diag(C)$ hold by construction. Then, we have
\begin{align}
    Z 1_n &= \sum_{j=1}^{k} \frac{1}{c_j} \Pi_j 1_n = \sum_{j=1}^{k} \frac{1}{c_j} c_j \pi_j = \sum_{j=1}^{k} \pi_j = 1_n, \\
      \diag(Z) &= \diag \big(\sum_{j=1}^{k} \frac{1}{c_j} \Pi_j \big) = \sum_{j=1}^{k} \frac{1}{c_j} \diag(\Pi_j) = \sum_{j=1}^{k} \frac{1}{c_j} \pi_j = X \diag(C^{-1}).
\end{align}
Using the positive semidefiniteness of $\Pi_j - \pi_j (\pi_j)^\top$ we have
\begin{align}
    Z - X C^{-1} X^\top & = \sum_{j=1}^k \frac{1}{c_j} \Pi_j - [\pi_1, \dots, \pi_k] C^{-1} [\pi_1, \dots, \pi_k]^\top \\
    & = \sum_{j=1}^k \frac{1}{c_j} \Pi_j - \sum_{j=1}^k \frac{1}{c_j} \pi_j (\pi_j)^\top = \sum_{j=1}^k \frac{1}{c_j} (\Pi_j - \pi_j (\pi_j)^\top) \succeq 0.
\end{align}
{Clearly, the objective function values coincide}.
Finally, we have to show that $f^\star_{\textrm{ML}} \geq f^\star_{\textrm{PW}}$. For any feasible solution $(Z, X)$ of problem \eqref{prob:mlSDP} we obtain a feasible solution $\bar{Z}$ of problem \eqref{prob:SDPpeng} by setting $\bar{Z} = Z$. Clearly, $\bar{Z} \geq 0_{n \times n}$ and $\bar{Z}1_n = 1_n$ hold by construction. Then, $\bar{Z} \succeq 0$ since $Z \succeq X C^{-1} X^\top$. Furthermore, constraints $\diag(Z) = X \diag(C^{-1})$ and $X^\top 1_n = \diag(C)$ imply
\begin{align}
    \tr(\bar{Z}) = 1_n^\top \diag(Z) = 1_n^\top X \diag(C^{-1}) = \diag(C) \diag(C^{-1}) = k
\end{align}
and hence $\bar{Z}$ is feasible for problem \eqref{prob:SDPpeng}.
\end{proof} 

To compare the performance of the two relaxations, we select small-scale benchmark instances from the constrained clustering literature \cite{rujeerapaiboon2019size, haouas2020exact} that are described in Section \ref{sec:results} (see Table \ref{tab:instances}). We solve both the SDP relaxations by using MOSEK interior-point optimizer \cite{mosek}. {We set a default tolerance of $10^{-8}$ for primal feasibility, dual feasibility, and relative gap.}
In Table \ref{tab:mosek}, along with the number of data points $n$ and the target number of cluster $k$, we report the lower bound (LB) provided by the relaxation \eqref{prob:vlSDP}, the relaxation \eqref{prob:mlSDP} and the computational time in seconds. 
\begin{table}[!ht]
    \centering
\begin{tabular}{|lcc|cc|cc|}
\hline
& & & \multicolumn{2}{c|}{VL-SDP} & \multicolumn{2}{c|}{ML-SDP}\\
\multicolumn{1}{|l}{Dataset} & $n$ & $k$ & LB &  Time [s] &  LB &  Time [s] \\
\hline
          Ruspini & 75 & 4 &    1.2881e+04 &        8.09 &     1.2881e+04 &        2.92 \\
     BreastTissue &  106 & 6 &    2.3710e+10 &        240.86 &     2.3709e+10 &        49.67 \\
Hierarchical & 118 & 4 &    7.3973e+06 &        194.13 &     7.3668e+06 &        49.84 \\
             Iris &   150 & 3 &      8.1278e+01 &        499.51 &         8.1278e+01 &        145.09 \\
     HapticsSmall &  155 & 5 &   1.7759e+04 &        1176.5 &     1.7622e+04 &        225.77 \\
            UrbanLand &  168 & 9 &   3.4317e+09 &        4226.19 &     3.4185e+09 &        495.64 \\
             Wine &  178 & 3 &   2.3983e+06 &        1202.17 &     2.3853e+06 &        438.14 \\
        Parkinson & 195 & 2 &    1.3641e+06 &        3286.55 &     1.3399e+06 &        723.83 \\
\hline
\end{tabular}
    \caption{Comparison between VL-SDP and ML-SDP bounds on small-scale instances using MOSEK. The cluster sizes are set according to the ground-truth class labels.}
    \label{tab:mosek}
\end{table}
The results show that our matrix lifting SDP relaxation provides competitive bounds and is solved significantly faster than the existing vector lifting SDP relaxation. Besides this, the results show that, on some instances, the relaxations provide bounds that are close, and for some problems even equal. Due to the mentioned quality of the new relaxation, we believe that it is suitable for implementation within a branch-and-cut framework. {Since SDP solvers can be inaccurate, we describe how to derive valid lower bounds for both SDPs so that they can be safely used within a branch-and-cut framework.}


\section{Valid Lower Bounds} \label{sec:sdp_safe}
Current successful solution techniques based on B\&B methods rely on obtaining strong and inexpensive bounds.
Computational results in Table \ref{tab:mosek} show that, regardless of the relaxation, off-the-shelf interior-point methods (IPMs) are inefficient even for small-sized clustering instances. Thus, for the efficiency of the B\&B we need algorithms that can solve large-scale SDPs. Compared to IPMs \cite{alizadeh1995interior}, first-order SDP solvers based on augmented Lagrangian method (ALM) or alternating direction method of multipliers (ADMM) can scale to significantly larger problem sizes, while trading off the accuracy of the resulting output \cite{wen2010alternating, sun2015convergent, yang2015sdpnal}.
When using first-order methods, it is hard to reach a solution to high precision in a reasonable amount of time. To safely use the proposed lower bounds within the B\&B algorithm, we need a certificate that the optimal value of the primal SDP is indeed a valid lower bound for the discrete optimization problem. However, since we only approximately solve the primal-dual pair of SDPs to some precision, feasibility is not necessarily reached when the algorithm terminates. In the following, we consider post-processing methods to guarantee safe lower bounds for our SDPs. 
Following the ideas developed in \cite{jansson2008rigorous} and \cite{cerulli2021improving}, we use two methods: one adding a negative perturbation to the dual objective function value (error bounds) and one that generates a dual feasible solution, and hence a bound, by solving a linear program. Both methods are computationally cheap and produce bounds close to the optimal objective function value. However, since in our experiments there is no method that systematically outperforms the other on all the considered instances, we choose to run both of them and we use the best (largest) bound.


Consider Lagrange multipliers $y \in \mathbb{R}^n$, $\alpha_j, v_j \in \mathbb{R}$, $\beta_j, \gamma_j, u_j \in \mathbb{R}^n$, $V_j \in \mathcal{S}^{n}$, $V_j \geq 0$, $U_j \in \mathcal{S}^{n+1} \ \forall j \in [k]$. The dual of problem \eqref{prob:vlSDP} is
{\small
\begin{maxi}[2]
{}{y^\top 1_n + \sum_{j=1}^k \alpha_j c_j - \sum_{j=1}^{k} v_j + {\ip{W}{I_n}}}
    {\label{prob:dualvlSDP}}{}
\addConstraint{-y - \alpha_j 1_n + \beta_j + c_j \gamma_j - 2 u_j}{= 0_n}{\quad \forall j \in [k]}
\addConstraint{{-\frac{1}{c_j}W} - \Diag(\beta_j) - \frac{1}{2} 1_n (\gamma_j)^\top - \frac{1}{2} \gamma_j 1_n^\top - S_j}{= V_j}{\quad \forall j \in [k]}
\addConstraint{U_j = \begin{bmatrix}
v_j & (u_j)^\top \\
u_j & S_j
\end{bmatrix}}{\in \mathcal{S}_{+}^{n+1}, \ V_j \geq 0_{n \times n}}{\quad \forall j \in [k]}
\end{maxi}}

Consider Lagrange multipliers $y_1, \alpha_1, \alpha_2 \in \mathbb{R}^n$, $y_2 \in \mathbb{R}^k$, $U \in \mathbb{R}^{n \times k}$, $U \geq 0$, $V \in \mathcal{S}^n$, $V \geq 0$, $S \in \mathcal{S}^{n+k}$. The dual of problem \eqref{prob:mlSDP} is
{\small
\begin{maxi}[2]
{}{y_1^\top 1_n + y_2^\top \diag(C) + \alpha_1^\top 1_n - \ip{C}{S_{11}} + {\ip{W}{I_n}}}
    {\label{prob:dualmlSDP}}{}
\addConstraint{- y_1 1_k^\top - 1_n y_2^\top + \alpha_2 \diag(C^{-1})^\top - 2S_{12}}{= U}
\addConstraint{{-W} - \frac{1}{2}\alpha_1 1_n^\top - \frac{1}{2} 1_n \alpha_1^\top - \Diag(\alpha_2) - S_{22}}{= V}
\addConstraint{S = \begin{bmatrix}
S_{11} & S_{12}^\top \\
S_{12} & S_{22}
\end{bmatrix} \in \mathcal{S}^{n+k}_+, }{\ U \geq 0_{n \times k}, \ {V \geq 0_{n \times n}}.}{}
\end{maxi}}

 We use the dual problems in both post-processing techniques described in the next subsections. The idea of post-processing via error bounds is the following. If the dual feasibility is reached within machine accuracy, then the dual objective function is already a valid lower bound. Otherwise, the dual objective value is perturbed by adding a negative term to keep into account the infeasibility. This perturbation should be as small as possible.
 Theorems \ref{theorem:pp_vl} and \ref{theorem:pp_ml} compute the tailored safe underestimate of the dual objective function for problems \eqref{prob:vlSDP} and \eqref{prob:mlSDP}, respectively. The following lemma is needed for proving the validity of the error bounds.
\begin{lemma}[Lemma 3.1 in ref. \cite{jansson2008rigorous}]\label{lem:jansson}
Let $S, X \in \mathcal{S}^n$ be matrices that satisfy
$0 \leq \lambda_{\min}(X)$
 and $\lambda_{\max}(X) \leq \bar{x}$ for some $\bar{x} \in \mathbb{R}$.
Then the following inequality holds:
\begin{equation*}
    \ip{S}{X} \geq \bar{x}\sum_{i \colon  \lambda_i(S) <0}\lambda_i(S).
\end{equation*}
\end{lemma}



\begin{theorem}
\label{theorem:pp_vl}
{Let $p^*$ be the optimal objective function value of \eqref{prob:vlSDP}. Given the dual variables $y \in \mathbb{R}^n$, $\alpha_j, v_j \in \mathbb{R}$, $\beta_j, \gamma_j \in \mathbb{R}^n$, $V_j \in \mathcal{S}^{n}$, $V_j \geq 0$, set
\begin{align*}
    U_j = \begin{bmatrix}
v_j & (u_j)^\top \\
u_j & S_j
\end{bmatrix}, \quad u_j &= \frac{1}{2}\left(-y - \alpha_j 1_n + \beta_j + c_j \gamma_j\right), \\
S_j &= -\frac{1}{c_j}W - \Diag(\beta_j) - \frac{1}{2} 1_n (\gamma_j)^\top - \frac{1}{2} \gamma_j 1_n^\top - V_j,
\end{align*}
for all $j \in [k]$}. A safe lower bound for $p^\star$ is given by
\begin{equation*}
    lb = {\ip{W}{I_n}} + y^\top 1_n + \sum_{j=1}^k \alpha_j c_j - \sum_{j=1}^{k} v_j + \sum_{j=1}^{k} \Big( \left(c_j+1\right) \sum_{i\colon \lambda_i(U_j) < 0} \lambda_i(U_j) \Big).
\end{equation*}
\end{theorem}
\begin{proof}
{Let $Y_j^\star = \begin{bmatrix}
1 & (\pi_j^\star)^\top \\
\pi_j^\star & \Pi_j^\star
\end{bmatrix}$ for all $j \in [k]$ be an optimal solution of \eqref{prob:vlSDP} with objective function value $p^\star$.}
In order to show that $p^\star \ge lb$, consider the equation
{\footnotesize
\begin{align*}
   & \quad \ {\ip{W}{I_n - \sum_{j=1}^k \frac{1}{c_j} \Pi_j^\star}} - \left(y^\top 1_n + \sum_{j=1}^k \alpha_j c_j + \sum_{j=1}^k (\beta_j)^\top 0_n + \sum_{j=1}^k (\gamma_j)^\top 0_n - \sum_{j=1}^{k} v_j + {\ip{W}{I_n}}\right) \\
   & = {-\ip{W}{\sum_{j=1}^k \frac{1}{c_j} \Pi_j^\star}} - y^\top \left(\sum_{j=1}^k \pi_j^\star\right) - \sum_{j=1}^k \alpha_j (1_n^\top \pi_j^\star) - \sum_{j=1}^k (\beta_j)^\top (\diag(\Pi_j^\star) - \pi_j^\star) \\
   & \quad - \sum_{j=1}^k (\gamma_j)^\top \left(\frac{1}{2}(\Pi_j^\star + (\Pi_j^\star)^\top) 1_n - c_j \pi_j^\star \right) + \sum_{j=1}^{k} v_j \\
   & = {\sum_{j=1}^k -\frac{1}{c_j} \ip{W}{\Pi_j^\star} }- \sum_{j=1}^k y^\top \pi_j^\star - \sum_{j=1}^k \alpha_j 1_n^\top \pi_j^\star - \ip{\sum_{j=1}^k \Diag(\beta_j)}{\Pi_j^\star} + \sum_{j=1}^k (\beta_j)^\top \pi_j^\star \\
   & \quad - \ip{\frac{1}{2} \sum_{j=1}^k 1_n (\gamma_j)^\top}{\Pi_j^\star} - \ip{\frac{1}{2} \sum_{j=1}^k \gamma_j 1_n^\top}{\Pi_j^\star} + \sum_{j=1}^k (\gamma_j)^\top c_j \pi_j^\star + \sum_{j=1}^{k} v_j \\
   & = \sum_{j=1}^k \ip{{-\frac{1}{ c_j} W }- \Diag(\beta_j) - \frac{1}{2} 1_n (\gamma_j)^\top - \frac{1}{2} \gamma_j 1_n^\top}{\Pi_j^\star} + \sum_{j=1}^k \left(- y - \alpha_j 1_n + \beta_j + \gamma_j c_j\right)^\top \pi_j^\star + \sum_{j=1}^k v_j \\
   & = \sum_{j=1}^k \ip{V_j}{\Pi_j^\star} + \sum_{j=1}^k \ip{S_j}{\Pi_j^\star} + 2 \sum_{j=1}^k (u_j)^\top \pi_j^\star + \sum_{j=1}^k v_j = \sum_{j=1}^k \ip{V_j}{\Pi_j^\star} + \sum_{j=1}^k \ip{U_j}{Y^\star_j}.
\end{align*}}
{The last term can be bounded by means of Lemma \ref{lem:jansson} applied on matrix $Y^\star_j$, so that we need an upper bound $\bar{y}_j$  on $\lambda_{\max}(Y^\star_j)$. The largest eigenvalue of $\Pi_j$ is lower or equal than its largest row sum, which is the cardinality of cluster $j$. Due to $\Pi_j 1_n = c_j \pi_j$ and $\pi_j \in [0, 1]^n$ we set $\bar{y}_j = c_j + 1$.  Therefore, using Lemma \ref{lem:jansson} with $\bar{y}_j = c_j+1$, and the nonnegativity of $V_j$ for all $j \in [k]$, we obtain
{\small
\begin{align*}
 \sum_{j=1}^k \ip{V_j}{\Pi_j^\star} + \sum_{j=1}^k \ip{U_j}{Y^\star_j} \geq \sum_{j=1}^k \Big(\left(c_j+1\right) \sum_{i\colon \lambda_i(U_j) < 0} \lambda_i(U_j) \Big).
\end{align*}
}}
\end{proof}
\noindent

In order to state an analogous result for problem \eqref{prob:mlSDP}, we need an intermediate result, bounding the eigenvalues of its feasible solutions.

\begin{lemma}[ref. \cite{BOURIN20121906}]\label{lemma:blockpsd}
For every block matrix $M = {\begin{bmatrix}
M_{11} & M_{12}^\top \\
M_{12} & M_{22}
\end{bmatrix}} \succeq 0$ we have the decomposition
\begin{equation*}
{M = 
U \begin{bmatrix}
M_{11} & 0 \\
0 & 0
\end{bmatrix} U^\top +
V \begin{bmatrix}
0 & 0 \\
0 & M_{22}
\end{bmatrix} V^\top}
\end{equation*}
{for some orthogonal matrices $U, V$.}
\end{lemma}

\begin{theorem}\label{theorem:eigboundY}
{Let $Y = \begin{bmatrix}
C & X^\top \\
X & Z
\end{bmatrix}$} be a feasible solution of problem \eqref{prob:mlSDP}, then 
\begin{equation*}
   \max_{i=1,\dots, k} C_{ii} \leq \lambda_{\max}(Y) \leq \max_{i=1,\dots, k} C_{ii} + 1
\end{equation*}
\end{theorem}
\begin{proof}
Let $C_0 = 
\begin{bmatrix}
C & 0 \\
0 & 0
\end{bmatrix}$ and $Z_0 = 
\begin{bmatrix}
0 & 0 \\
0 & Z
\end{bmatrix}$. From Lemma \ref{lemma:blockpsd} we have 
\begin{equation*}
Y = 
U C_0 U^\top +
V Z_0 V^\top,
\end{equation*}
where $U$ and $V$ are orthogonal. Therefore
\begin{align*}
\lambda_{\max}(Y) = \max_{\|v\|=1} v^\top Y v &= \max_{\|v\|=1} v^\top \big(U C_0 U^\top +
V Z_0 V^\top \big) v \\
&\leq \max_{\|v\|=1} v^\top \big(U C_0 U^\top \big) v + \max_{\|v\|=1} v^\top \big(
V Z_0 V^\top \big) v \\
&= \lambda_{\max}\big( U C_0 U^\top\big) + \lambda_{\max}\big( V Z_0 V^\top\big) \\
&= \lambda_{\max}\big(C_0 \big) + \lambda_{\max}\big( Z_0 \big) \\
&= \lambda_{\max}(C) + \lambda_{\max}(Z) = \max_{i=1,\dots,k} C_{ii} + 1,
\end{align*}
where the last two equations derive from the orthogonality of $U$ and $V$ and $Z$ being a stochastic matrix. 
Finally,
\begin{align*}
    \lambda_{\max}(Y) = \max_{\|v\|=1} v^\top Y v \geq e_i^\top Y e_i = y_{ii} \quad \forall i \in \{1, \dots, n+k\},
\end{align*}
where $e_i$ is the $i$-th basis vector. Therefore, $\lambda_{\max}(Y) \geq \max_{i=1,\dots,k} C_{ii}$.
\end{proof}

\noindent

\noindent
Similarly to the vector lifting relaxation, once the SDP has been solved approximately, the following theorem gives a lower bound on the optimal value of the matrix lifting SDP relaxation.
\begin{theorem}
\label{theorem:pp_ml}
{Let $p^\star$ be the optimal objective function value of problem \eqref{prob:mlSDP}.Given the dual variables $y_1, \alpha_1, \alpha_2 \in \mathbb{R}^n$, $y_2 \in \mathbb{R}^k$, $U \in \mathbb{R}^{n \times k}$, $U \geq 0$, $V \in \mathcal{S}^n$, $V \geq 0$, $S_{11} \in \mathbb{R}^{k \times k}$ set
\begin{align*}
    S = \begin{bmatrix}
S_{11} & S_{12}^\top \\
S_{12} & S_{22}
\end{bmatrix}, \quad S_{12} &= \frac{1}{2}\left(- y_1 1_k^\top - 1_n y_2^\top + \alpha_2 \diag(C^{-1})^\top - U \right), \\
S_{22} &= -W - \frac{1}{2}\alpha_1 1_n^\top - \frac{1}{2} 1_n \alpha_1^\top - \Diag(\alpha_2) - V.
\end{align*}} A safe lower bound for the optimal value $p^\star$ is given by
\begin{equation*}
    lb = {\ip{W}{I_n}} + y_1^\top 1_n + y_2^\top \diag(C) + \alpha_1^\top 1_n - \ip{C}{S_{11}} + \left(\max_{i=1, \dots, k} C_{ii} + 1 \right)\sum_{i\colon \lambda_i(S) < 0} \lambda_i(S).
\end{equation*}
\end{theorem}
\begin{proof}
{Let $Y^\star = \begin{bmatrix}
C & (X^\star)^\top \\
X^\star & Z^\star
\end{bmatrix}$ be an optimal solution of problem \eqref{prob:mlSDP} with objective function value $p^\star$. }
In order to show that $p^\star \ge lb$, consider the equation
\begin{subequations}
\begin{align*}
    {-\ip{W}{Z^\star} + \ip{W}{I_n}} & - \big( y_1^\top 1_n + y_2^\top \diag(C) + \alpha_1^\top 1_n + \alpha_2^\top 0_n - \ip{C}{S_{11}} + {\ip{W}{I_n}} \big) \\
    & = {-\ip{W}{Z^\star}} - y_1^\top X^\star 1_k - y_2^\top (X^\star)^\top 1_n - \frac{1}{2}\alpha_1^\top (Z^\star + (Z^\star)^\top) 1_n \\
    & - \alpha_2^\top(\diag(Z^\star) - X^\star \diag(C^{-1}))  + \ip{C}{S_{11}} \\
    & = {-\ip{W}{Z^\star}} - \ip{ y_1 1_k^\top}{X^\star} - \ip{1_n y_2^\top}{X^\star} - \frac{1}{2}\ip{ 1_n \alpha_1^\top}{Z^\star} \\
    & - \frac{1}{2} \ip{\alpha_1 1_n^\top}{Z^\star} - \ip{\Diag(\alpha_2)}{Z^\star} + \ip{\alpha_2 \diag(C^{-1})^\top}{X^\star} +\ip{C}{S_{11}} \\
    & = \ip{{-W} - \frac{1}{2} 1_n \alpha_1^\top - \frac{1}{2} \alpha_1 1_n^\top - \Diag(\alpha_2)}{Z^\star} \\
    & + \ip{-y_1 1_k^\top - 1_n y_2^\top + \alpha_2 \diag(C^{-1})^\top}{X^\star} + \ip{C}{S_{11}} \\
    & = \ip{V}{Z^\star} + \ip{U}{X^\star} + \ip{S_{11}}{C} + 2\ip{S_{12}}{X^\star} + \ip{S_{22}}{Z^\star}\\
    & = \ip{V}{Z^\star} + \ip{U}{X^\star} + \ip{S}{Y^\star}.
\end{align*}
\end{subequations}
{The last term can be bounded by using Lemma \ref{lem:jansson} applied on $Y^\star$, with $\bar{y} \geq \lambda_{\max}(Y^\star)$. Here, a suitable bound $\bar{y}$ for the maximum eigenvalue of any feasible solution $Y$ can be obtained by applying Theorem \ref{theorem:eigboundY}. Therefore, we set $\bar{y} = \max_{i=1, \dots, k} C_{ii} + 1$, and applying Lemma \ref{lem:jansson} we get 
\begin{align*}
\ip{V}{Z^\star} + \ip{U}{X^\star} + \ip{S}{Y^\star} & \geq \ip{V}{Z^\star} + \ip{U}{X^\star} + \left(\max_{i=1, \dots, k} C_{ii} + 1\right) \sum_{i\colon \lambda_i(S) < 0} \lambda_i(S) \\
& \geq \left(\max_{i=1, \dots, k} C_{ii} + 1\right) \sum_{i\colon \lambda_i(S) < 0} \lambda_i(S),
\end{align*}
where the last inequality holds because $U$ and $V$ are nonnegative.}
\end{proof}
\noindent


Another way to get valid lower bounds for problems \eqref{prob:vlSDP} and \eqref{prob:mlSDP} is to tune the output results to get a feasible solution for the dual SDP problem. More in detail, from an approximate dual solution, that is not positive semidefinite, a partial dual feasibile solution is obtained by computing its projection onto the positive semidefinite cone. Then, the remaining dual variables are computed by solving an auxiliary LP. If the LP is feasible, we can construct a feasible solution, obtaining a valid dual bound. If it is infeasible, then we are neither able to construct a feasible dual solution nor to construct a dual bound. In this case, we set $lb = - \infty$ and we use the valid bound provided by Theorems \ref{theorem:pp_vl} and \ref{theorem:pp_ml}, respectively.

\section{Cutting-plane Algorithm}
\label{sec:cutting_plane}
The bounds given by relaxations \eqref{prob:vlSDP} and \eqref{prob:mlSDP} are not strong enough to be successfully used within a B\&B framework to solve large ccMSSC problems to optimality. We propose to strengthen the SDP bounds by using polyhedral cuts and solve the resulting SDPs by means of a cutting-plane algorithm. 
As mentioned in \cite{rujeerapaiboon2019size}, trivial inequalities for the vector lifting relaxation are obtained via the reformulation-linearization technique (RLT).
In \cite{rujeerapaiboon2019size}, all RLT inequalities are included at once. We did some experiments adding all the RLT constraints (for small instances) or adding only the violated ones in a cutting-plane fashion. However, the improvement of the bound was marginal and came at a high computational cost. 
This motivated us to look for more effective valid inequalities. To this end, one may add any inequality that is valid for the so-called boolean quadric polytope, which is defined in \cite{padberg1989boolean} as the convex hull of pairs $(\Pi_j, \pi_j)$ satisfying $\Pi_j = \pi_j(\pi_j)^\top$ with $\pi_j \in \{0, 1\}^n$ for $j \in [k]$. One important class of these cuts are the triangle inequalities, defined as
\begin{equation}
\label{eq:triangle_vl}
\begin{split}
    (\pi_j)_r + (\pi_j)_s + (\pi_j)_t & \leq (\Pi_j)_{rs} + (\Pi_j)_{rt} + (\Pi_j)_{st} + 1 \quad \forall j \in [k],\\
    (\Pi_j)_{rs} + (\Pi_j)_{rt} & \leq (\pi_j)_r + (\Pi_j)_{st} \quad \forall j \in [k]
\end{split}
\end{equation}
which hold for all distinct triplets $(r, s, t)$. Note that there are $O(k \binom{n}{3})$ inequalities of type \eqref{eq:triangle_vl}.
Triangle inequalities explore the following property of the clustering problem: if data points $p_r$, $p_s$ and $p_r$, $p_t$ belong to the cluster $j$, then $p_s$ and $p_t$ must be in the cluster $j$ as well.
For the matrix lifting SDP, we consider triangle inequalities of the form
\begin{equation}
    \label{eq:triangle_ml}
    \begin{split}
            Z_{ij} & \leq Z_{ii},\\
            Z_{ij} + Z_{ih} & \leq Z_{ii} + Z_{jh}
    \end{split}
\end{equation}
which hold for all distinct triplets $(i, j, h)$. Constraints \eqref{eq:triangle_ml} are also used to strengthen the LP relaxation in \cite{aloise2009branch} and the SDP relaxation in \cite{piccialli2022sos} for the unconstrained MSSC. Note that there are $O(\binom{n}{3})$ inequalities of type \eqref{eq:triangle_ml}.
Similarly to \eqref{eq:triangle_vl}, these inequalities ensure that if $p_i$, $p_j$ and $p_i, p_h$ are in the same cluster, then $Z_{ii} = Z_{ij} = Z_{ih} = Z_{jh}$ must hold.


The enumeration of all triangle inequalities is computationally inexpensive, even for large instances. However, adding all of them
would make the relaxation intractable even for moderate values of $n$. To keep both the SDPs for the bounding routine and the auxiliary LPs for post-processing to a modest size, we limit the number of triangle inequalities that we can add at each cutting-plane iteration. The triangle inequalities are then sorted by violation magnitude and added starting with the most violated ones.
Once the SDP minimizer is obtained, we remove all inactive constraints and add new violated inequalities. Next, the problem with an updated set of inequalities is solved and the process is iterated as long as the increase of the lower bound is sufficiently large. Indeed, we terminate the bounding routine if the relative difference between consecutive bounds is less than a fixed threshold. If the gap cannot be closed after some cutting plane iterations, we terminate the bounding routine, branch the current node, and start solving new subproblems as described in Section \ref{sec:branching}. Note that, to improve the efficiency of the overall B\&B algorithm, we pass cutting planes from parent to children nodes. This allows us to save a significant number of cutting-plane iterations, and therefore computational time, when processing children nodes. We also point out that, when an SDP is solved with a new set of inequalities, the post-processing phase yielding valid lower bounds is performed by adapting the results in Section \ref{sec:sdp_safe}.

\section{Branching}
\label{sec:branching}
Using cutting planes to tighten the SDP relaxations may not be sufficient to prove the optimality of the best clustering solution found so far.
If there are no violated cuts or the bound does not improve significantly when adding valid inequalities, then we stop the generation of cutting planes and we branch, i.e., we split the current problem into more problems of smaller dimensions by fixing some variables. 
An appropriate branching strategy can help limit the branching decisions and prevent unnecessary bound computations. In the notation of problem \eqref{prob:ccMSSC1}, one may selects a variable $(\pi_j)_i$ having a nonbinary value and produce two subproblems: one with $(\pi_j)_i = 1$ (i.e., the $i$-th data point is assigned to the cluster $j$) and the other with $(\pi_j)_i = 0$ (i.e., the $i$-th data point is not in the cluster $j$). If a subproblem contains additional cluster indicator variables with nonbinary values then this process can be continued to produce a branching tree that will eventually enforce all binary conditions. 
Despite the existence of many branching rules that perform well for generic problems, there also exist branching rules tailored to the specific MSSC problem. Following \cite{piccialli2022exact}, we adopt a branching logic aimed at enforcing pairwise constraints, namely \textit{must-link} and \textit{cannot-link} constraints. More in detail, a must-link constraint is used to specify that the two data points are in the same cluster, whereas a cannot-link constraint states that they can not be placed together. Both types of constraints can be added to the SDP relaxations through linear equality and inequality constraints. To this end, let $[l_{=}] \subset [n] \times [n]$ and $[l_{\neq}] \subset [n] \times [n]$ be the set of must-link and cannot-link constraints, respectively. 
Consider the vector lifting relaxation and define the following sets

{\footnotesize
\begin{align*}
    \mathcal{ML_{\textrm{VL}}} &= \left\{(\pi, \Pi) \in \mathbb{R}^n \times \mathcal{S}^n : \forall (i, j) \in [l_{=}] \ \begin{array}{l}
    \pi_i = \pi_j, \ \Pi_{it} = \Pi_{jt} \ \forall t \in [n]
    \end{array}\right\}, \\
    \mathcal{CL_{\textrm{VL}}} &= \left\{(\pi, \Pi) \in \mathbb{R}^n \times \mathcal{S}^n : \forall (i, j) \in [l_{\neq}] \ \begin{array}{l}
    \pi_i + \pi_j \leq 1, \ \Pi_{ij} = 0
    \end{array}\right\}.
\end{align*}
}%
At any level of the search tree, the vector lifting relaxation with pairwise constraints becomes
\begin{mini}[2]
{}{{\ip{W}{I_n-\sum_{j=1}^{k} \frac{1}{c_j} \Pi_j}}}
    {\label{prob:vlSDP_ml_cl}}{}
\addConstraint{\sum_{j=1}^{k} \pi_j}{= 1_n}{\quad \forall j \in [k]}
\addConstraint{(\pi_j, \Pi_j)}{\in \mathcal{S}_{\textrm{VL}}(c_j) \cap \mathcal{ML}_{\textrm{VL}} \cap \mathcal{CL}_{\textrm{VL}}}{\quad \forall j \in [k]}
\end{mini}
Similarly, consider the matrix lifting relaxation and define the analogues sets
{\footnotesize
\begin{align*}
\mathcal{ML_{\textrm{ML}}} &= \left\{(X, Z) \in \mathbb{R}^{n \times k} \times \mathbb{S}^n : \forall (i, j) \in [l_{=}] \ \begin{array}{l}
    X_{ih} = X_{jh} \ \forall h \in [k], \ Z_{ih} = Z_{jh} \ \forall h \in [n]
    \end{array}\right\},\\ 
    \mathcal{CL_{\textrm{ML}}} &= \left\{(X, Z) \in \mathbb{R}^{n \times k} \times \mathbb{S}^n : \forall (i, j) \in [l_{\neq}] \ \begin{array}{l}
    X_{ih} + X_{jh} \leq 1 \ \forall h \in [k], \ Z_{ij} = 0
    \end{array}\right\}.
\end{align*}
}%
Therefore, at any level of the tree, the problem takes the form
\begin{mini}[2]
{}{{\tr \left(W - W Z \right)}}
    {\label{prob:mlSDP_ml_cl}}{}
\addConstraint{(X, Z)}{\in \mathcal{S}_{\textrm{ML}} \cap \mathcal{ML}_{\textrm{ML}} \cap \mathcal{CL}_{\textrm{ML}}}{}.
\end{mini}
Pairwise constraints partition the set of feasible solutions associated to the parent node into two disjoint subsets. Hence, our algorithm generates a binary enumeration tree: every time a parent node is split into two children, a pair of indices $(i^\star, j^\star)$ is chosen and a must-link and a cannot-link constraint are imposed on their associated problems. Depending on the relaxation, we use the following rules for selecting the branching pair:
{\small
\begin{align*}
(i^\star, j^\star)_{\textrm{VL}} & = \argmax_{i, j=1, \dots n} \left\{ \min_{h = 1, \dots, k} \left\{(\Pi_h)_{ij}, \| (\Pi_h)_{i,:}-(\Pi_h)_{j,:} \|_2^2 \right\} \right\},\\
(i^\star, j^\star)_{\textrm{ML}} & = \argmax_{i, j = 1, \dots, n} \left\{ \min \left\{Z_{ij}, \| Z_{i,:}-Z_{j,:} \|_2^2 \right\} \right\}.
\end{align*}
}
In other words, we choose indices $i^\star$ and $j^\star$ with the least tendency to assign data points $p_{{i}^\star}$ and $p_{{j}^\star}$ to the same cluster, or to different ones. 

\begin{remark}
A branching rule based on pairwise constraints becomes relevant in our context under three different aspects.
First, it allows to constrain multiple variables involving two data points at the same time. This yields tighter relaxations in the children compared to the strategy of enforcing integrality on the variable of a single data point. Second, it allows to exploit symmetry reduction techniques that can significantly reduce the size of the SDPs. Third, it leads to the natural extension of SDP relaxations for ccMSSC to a general class of MSSC with both pairwise and cardinality constraints. To the best of our knowledge, this is the first time that tractable conic formulations are proposed for the MSSC problem with side constraints \cite{liberti2021side}.
\end{remark}

The number of pairwise constraints grows exponentially in the number of B\&B nodes. As a result, the bounding problems may become expensive to solve. However, must-link constraints can be exploited to reduce the size of the subproblems. To this end, by introducing a suitable transformation matrix that maps the data points onto the so-called ``super points'', problems \eqref{prob:vlSDP_ml_cl} and \eqref{prob:mlSDP_ml_cl} can be reformulated as programs on lower dimensional positive semidefinite cones. In the next section, we show that the cardinality constraints can be mapped onto the set of super points.

\subsection{Size reduction}
We consider the undirected graph $G = (V, E)$ where the vertices are the data points and the edges are the must-link constraints, i.e., $V = [n]$ and $E = [l_{=}]$. We find the connected components of $G$ and we assume that there are $m \leq n$ components $[B_1], \dots, [B_m]$. 
Let $T^s$ be the $m \times n$ transformation matrix having $T^s_{ij} = 1$ if $j \in [B_i]$ and 0 otherwise, for each $i \in \{1, \dots, m\}$. Matrix $T^s$ encodes the data points belonging to the same connected component
and vector $e^s = T^s 1_n$ contains the number of data points in each component. Observe that $T^s W (T^s)^\top$ reduces the size of $W$. As a result, instead of clustering the initial set of data points $p_1, \dots, p_n$, we find a $k$-partition of the super points $s_1, \dots, s_m$. In addition to removing must-link constraints, we are able to reduce the number of cannot-link constraints. To see this, we denote by $[l^s_{\neq}]$ the set of cannot-link constraints between super points. A cannot-link is defined on two super points $s_i$ and $s_j$ if there exists a cannot-link on two data points $p$ and $q$ such that $p \in [B_i]$ and $q \in [B_j]$. It is easy to verify that, after this process, cannot-link constraints between data points in different connected components are mapped to the same pair of super points.
\noindent
Hence, we can reduce the size of problem \eqref{prob:vlSDP_ml_cl} at any level of the tree reformulating it as
\begin{mini!}[2]
{}{\tr \left(W - T^{s} W (T^{s})^\top \sum_{j=1}^{k} \frac{1}{c_j} \Pi_j^s \right)}
    {\label{prob:vlSDP_shr}}{}
\addConstraint{\sum_{j=1}^{k} \pi^s_j}{= 1_m}{}
\addConstraint{(e^s)^\top \pi^s_j}{= c_j}{\quad \forall j \in [k]}
\addConstraint{(\pi^s_h)_i + (\pi^s_h)_j}{\leq 1}{\quad \forall h \in [k], \ \forall (i, j) \in [l^s_{\neq}]}
\addConstraint{\begin{bmatrix}
1 & (\pi^s_j)^\top \\
\pi^s_j & \Pi^s_j
\end{bmatrix}}{\in \mathcal{S}_{+}^{m+1}}{\quad \forall j \in [k]}
\addConstraint{(\Pi^s_h)_{ij}}{= 0}{\quad \forall h \in [k], \ \forall (i, j) \in [l^s_{\neq}]}
\addConstraint{\label{eq:vl_shr_diag_PI} \diag(\Pi^s_j)}{= \pi^s_j}{\quad \forall j \in [k]}
\addConstraint{\Pi^s_j e^s}{= c_j \pi^s_j}{\quad \forall j \in [k]}
\addConstraint{\Pi^s_j}{\geq 0_{m \times m}}{\quad \forall j \in [k]}.
\end{mini!}
\begin{theorem}\label{th:vl_reduction}
Problems \eqref{prob:vlSDP_ml_cl} and \eqref{prob:vlSDP_shr} are equivalent.
\end{theorem}
\begin{proof}
Let $\{(\pi_j^s, \Pi_j^s)\}_{j=1}^k$ be a feasible solution of problem \eqref{prob:vlSDP_shr}. For all $j \in [k]$ define $\Pi_j = (T^s)^\top     \Pi_j^s T^s$ and $\pi_j = (T^s)^\top     \pi_j^s$.
This is equivalent to expanding matrices $\{\Pi_j^s\}_{j=1}^{k}$ and vectors $\{\pi_j^s\}_{j=1}^{k}$ by replicating the rows according to the indices of data points involved in must-link constraints. Therefore, for all $(s, t) \in [l]_{=}$, constraints $(\pi_j)_s = (\pi_j)_t$ and $(\Pi_j)_{sh} = (\Pi_j)_{th} \ \forall h \in [n]$ hold by construction. From the definition of $T^s$, $\Pi_j \geq 0_{n \times n}$ holds for all $j \in [k]$. Likewise, from \eqref{eq:vl_shr_diag_PI} constraints $\diag(\Pi_j) = \pi_j$ hold as well. Moreover, for all $j \in [k]$ we have
\begin{align*}
    \sum_{j=1}^{k} \pi_j &= \sum_{j=1}^{k} (T^s)^\top \pi_j^s = (T^s)^\top \sum_{j=1}^{k} \pi_j^s = (T^s)^\top 1_m = 1_n, \\
    1_n^\top \pi_j &= 1_n^\top (T^s)^\top \pi_j^s = (e^s)^\top \pi_j^s = c_j,\\
    \Pi_j 1_n &= (T^s)^\top \Pi_j^s T^s 1_n = (T^s)^\top \Pi_j^s e^s = (T^s)^\top \pi_j^s c_j = c_j \pi_j,
\end{align*}
and
\begin{align*}
    \Pi_j - \pi_j \pi_j^\top \succeq 0
    & \Leftrightarrow \ip{\Pi_j u}{u} - \ip{\pi_j \pi_j^\top u}{u} \\
    & = \ip{(T^s)^\top \Pi_j^s T^s u}{u} - \ip{(T^s)^\top \pi_j^s (\pi_j^s)^\top T^s u}{u} \\
    & = \ip{\Pi_j^s \bar{u}}{\bar{u}} - \ip{ \pi_j^s (\pi_j^s)^\top \bar{u}}{\bar{u}} \geq 0 \quad \forall \bar{u} = T^s u, \ u \neq 0_n. \\
    & \Leftrightarrow     \Pi_j^s - \pi_j^s(\pi_j^s)^\top \succeq 0.
\end{align*}
Finally, we have
\begin{equation*}
{\ip{W}{I_n-\sum_{j=1}^{k} \frac{1}{c_j} \Pi_j} = \ip{W}{I_n -(T^s)^\top \sum_{j=1}^{k} \frac{1}{c_j} \Pi_j^s T^s}
= \tr \left(W - T^{s} W (T^{s})^\top \sum_{j=1}^{k} \frac{1}{c_j} \Pi_j^s \right)
}
\end{equation*}
and thus $\{(\Pi_j, \pi_j)\}_{j=1}^k$ is a feasible solution of problem \eqref{prob:vlSDP_ml_cl} and the values of the objective functions coincide.

Now assume that  $\{(\Pi_j, \pi_j)\}_{j=1}^k$ is a feasible solution of problem \eqref{prob:vlSDP_ml_cl}. For all $j \in [k]$ define $\Pi_j^s = \Diag(1_m / e^s) T^s \Pi_j (T^s)^\top \Diag(1_m / e^s)$ and $\pi_j^s = \Diag(1_m / e^s) T^s \pi_j$.
Constraints \eqref{eq:vl_shr_diag_PI} hold by construction since $\Pi_j^s$ and $\pi_j^s$ are lower dimensional variables obtained from $\Pi_j$ and $\pi_j$. If $\Pi_j$ is nonnegative, then so is $\Pi_j^s$.
Furthermore, we have
\begin{align*}
    \sum_{j=1}^k \pi_j^s &= \sum_{j=1}^k \Diag(1_m / e^s) T^s \pi_j = \Diag(1_m / e^s) T^s 1_n = \Diag(1_m / e^s) e^s = 1_m, \\
    (e^s)^\top \pi_j^s &= (e^s)^\top \Diag(1_m / e^s) T^s \pi_j = \pi_j^\top (T^s)^\top \Diag(1_m / e^s) e^s = \pi_j^\top 1_n = c_j, \\
    \Pi_j^s e^s &= \Diag(1_m / e^s) T^s \Pi_j (T^s)^\top \Diag(1_m / e^s) e^s
    =  \Diag(1_m / e^s) T^s \Pi_j (T^s)^\top 1_m \\
    & =  \Diag(1_m / e^s) T^s \Pi_j 1_n = \Diag(1_m / e^s) T^s \pi_j c_j = \pi_j^s c_j,
\end{align*}
and
\begin{align*}
    \Pi_J^s - \pi_j^s (\pi_j^s)^\top \succeq 0
    & \Leftrightarrow \ip{\Pi_j^s u}{u} - \ip{\pi_j^s (\pi_j^s)^\top u}{u} \\
    & = \ip{\Diag(1_m / e^s) T^s \Pi_j (T^s)^\top \Diag(1_m / e^s) u}{u} \ + \\
    & - \ip{\Diag(1_m / e^s) T^s \pi_j \pi_j^\top (T^s)^\top \Diag(1_m / e^s) u}{u} \\
    & = \ip{\Pi_j \bar{u}}{\bar{u}} - \ip{\pi_j \pi_j^\top \bar{u}}{\bar{u}} \geq 0 \quad \forall \bar{u} = (T^s)^\top \Diag(1_m / e^s) u\\
    & \Leftrightarrow     \Pi_j - \pi_j \pi_j^\top \succeq 0.
\end{align*}
It remains to show that the objective function values coincide:
{\footnotesize
\begin{align*}
\tr(W) - \ip{T^s W (T^s)^\top}{\sum_{j=1}^k  \frac{1}{c_j} \Pi_j^s} 
&= 
\tr(W) - \ip{T^{s} W (T^{s})^\top}{\Diag(1_m / e^s) T^{s} \sum_{j=1}^k \frac{1}{c_j} \Pi_j (T^{s})^\top \Diag(1_m / e^s)}\\
&= 
\tr(W) - \ip{ W }{(T^{s})^\top  \Diag(1_m / e^s) T^{s}\sum_{j=1}^{k} \frac{1}{c_j}\Pi_j (T^{s})^\top \Diag(1_m / e^s) T^{s}}\\
&= 
\tr(W) - \ip{W}{\sum_{j=1}^k  \frac{1}{c_j} \Pi_j} = \ip{W}{I-\sum_{j=1}^{k} \frac{1}{c_j} \Pi_j}.
\end{align*}}
Note that $(T^{s})^\top \Diag(1_m / e^s) T^{s} \sum_{j=1}^k  \frac{1}{c_j} \Pi_j (T^{s})^\top \Diag(1_m / e^s) T^{s}$ ``averages'' over the rows of matrix $\sum_{j=1}^k  \frac{1}{c_j} \Pi_j$. Since these rows are identical due to must-link constraints, the last equation holds.
\end{proof}

Similarly, we can reduce the size of problem \eqref{prob:mlSDP_ml_cl} at any level of the tree reformulating it as
\begin{mini!}[2]
{}{{\tr (W - T^{s} W (T^{s})^\top}{Z^{s}})}
    {\label{prob:mlSDP_shr}}{}
\addConstraint{X^s 1_k}{= 1_m}{}
\addConstraint{(X^s)^\top e^s}{= \diag(C)}{}
\addConstraint{X^s_{ih} + X^s_{jh}}{\leq 1 }{\quad \forall h \in [k], \ \forall (i,j) \in [l^s_{\neq}]}
\addConstraint{{Z^s e^s}}{= 1_m}{}
\addConstraint{\label{eq:diagZshr} \diag(Z^s)}{= X^s \diag(C^{-1})}{}
\addConstraint{Z^s_{ij}}{= 0 }{\quad \forall (i,j) \in [l^s_{\neq}]}
\addConstraint{\label{eq:SDPblockshr} \begin{bmatrix}
C & (X^s)^\top \\
X^s & Z^s
\end{bmatrix} \in \mathcal{S}^{m+k}_{+}}{}{}{}
\addConstraint{X^s \geq 0_{m \times k}, \ Z^s \geq 0_{m \times m}}{}{}
\end{mini!}

\begin{theorem}\label{th:ml_reduction}
Problems \eqref{prob:mlSDP_ml_cl} and \eqref{prob:mlSDP_shr} are equivalent.
\end{theorem}
\begin{proof}
Let $(Z^s, X^s)$ a feasible solution of problem \eqref{prob:mlSDP_shr}. Define $Z = (T^s)^\top Z^s T^s$ and $X = (T^s)^\top X^s$.
This is equivalent to expanding matrices $Z^s$ and $X^s$ by replicating the rows according to the indices of data points involved in must-link constraints. Therefore, $X_{ih} = X_{jh}$ for all $h \in [k]$ and $Z_{ih} = Z_{jh}$ for all $h \in [n]$ hold by construction. Clearly, $Z \geq 0_{n \times n}$ and $X \geq 0_{n \times k}$ hold as well. Likewise, from $\diag(Z^s) = X^s \diag(C^{-1})$ it is easy to verify that $\diag(Z) = X \diag(C^{-1})$ holds. Moreover, we have that
\begin{align*}
    X1_k &= (T^s)^\top X^s 1_k = (T^s)^\top 1_m = 1_n, \\
    X^\top 1_n &= (X^s)^\top T^s 1_n = (X^s)^\top e^s = \diag(C), \\
    Z1_n &= (T^s)^\top Z^s T^s 1_n = (T^s)^\top Z^s e^s = (T^s)^\top 1_m = 1_n.
\end{align*}
Constraint \eqref{eq:SDPblockshr} can be rewritten as $\ip{Z^s v}{v} - \ip{X^s C^{-1} (X^s)^\top v}{v} \geq 0$ for all $v \neq 0_m$.
Therefore, we have
\begin{align*}
    Z - X C^{-1} X^\top \succeq 0
    & \Leftrightarrow \ip{Z u}{u} - \ip{X C^{-1} X^\top u}{u} \\
    & = \ip{(T^s)^\top Z^s T^s u}{u} - \ip{(T^s)^\top X^s C^{-1} (X^s)^\top T^s u}{u} \\
    & = \ip{Z^s \bar{u}}{\bar{u}} - \ip{ X^s C^{-1} (X^s)^\top \bar{u}}{\bar{u}} \geq 0 \quad \forall \bar{u} = T^s u, \ u \neq 0_n. \\
    & \Leftrightarrow     Z^s - X^s C^{-1} (X^s)^\top \succeq 0.
\end{align*}
Furthermore, $\ip{W}{Z} = \ip{W}{(T^s)^\top Z^s T^s}
= \ip{(T^s)W(T^s)^\top}{Z^s}$
and thus $(Z, X)$ is a feasible solution of problem \eqref{prob:mlSDP_ml_cl} and the values of the objective functions coincide.

\bigskip
\noindent
Assume that $(Z, X)$ is a feasible solution of problem \eqref{prob:mlSDP_ml_cl}, set $X^s = \Diag(1_m / e^s) T^s X$ and $Z^s = \Diag(1_m / e^s) T^s Z (T^s)^\top \Diag(1_m / e^s)$.
Constraint \eqref{eq:diagZshr} follows by construction since $Z^s$ and $X^s$ are the ``shrunk'' matrices of $Z$ and $X$. If $(Z, X)$ is nonnegative, then so is $(Z^s, X^s)$.
Furthermore, we can derive 
\begin{align*}
    X^s 1_k &= \Diag(1_m / e^s) T^s X 1_k = \Diag(1_m / e^s) T^s 1_n = \Diag(1_m / e^s) e^s = 1_m, \\
    (X^s)^\top e^s &= X^\top (T^s)^\top \Diag(1_m / e^s) e^s = (X^s)^\top e^s = X^\top (T^s)^\top 1_m = X^\top 1_n = \diag(C), \\
    Z^s e^s &= \Diag(1_m / e^s) T^s Z (T^s)^\top \Diag(1_m / e^s) e^s
    =  \Diag(1_m / e^s) T^s Z (T^s)^\top 1_m \\
    & =  \Diag(1_m / e^s) T^s Z 1_n
    =  \Diag(1_m / e^s) T^s 1_n
    =  \Diag(1_m / e^s) e^s
    =  1_m.
\end{align*}
Then, we have
\begin{align*}
    Z^s - X^s C^{-1} (X^s)^\top \succeq 0
    & \Leftrightarrow \ip{Z^s u}{u} - \ip{X^s C^{-1} (X^s)^\top u}{u} \\
    & = \ip{\Diag(1_m / e^s) T^s Z (T^s)^\top \Diag(1_m / e^s) u}{u} \ + \\
    & - \ip{\Diag(1_m / e^s) T^s X C^{-1} X^\top (T^s)^\top \Diag(1_m / e^s) u}{u} \\
    & = \ip{Z \bar{u}}{\bar{u}} - \ip{X C^{-1} X^\top \bar{u}}{\bar{u}} \geq 0 \quad \forall \bar{u} = (T^s)^\top \Diag(1_m / e^s) u\\
    & \Leftrightarrow     Z - X C^{-1} X^\top \succeq 0.
\end{align*}
It remains to show that the objective function values coincide:
\begin{align*}
\ip{T^s W (T^s)^\top}{Z^s} 
&= 
\ip{T^{s} W (T^{s})^\top}{\Diag(1_m / e^s) T^{s} Z (T^{s})^\top \Diag(1_m / e^s)}\\
&= 
\ip{ W }{(T^{s})^\top \Diag(1_m / e^s) T^{s} Z (T^{s})^\top \Diag(1_m / e^s) T^{s}} = \ip{W}{Z}.
\end{align*}
In the latter equation, note that $(T^{s})^\top \Diag(1_m / e^s) T^{s} Z (T^{s})^\top \Diag(1_m / e^s) T^{s}$ ``averages'' over selected rows of matrix $Z$. Since these rows are identical due to must-link constraints, the last equation holds.
\end{proof}

\section{Heuristic}
\label{sec:heuristic}
\noindent
In the previous sections we discussed how to obtain strong and inexpensive lower bounds for ccMSSC by using semidefinite programming tools.
In this section, we develop a rounding algorithm that recovers a feasible clustering (and thus an upper bound on the ccMSSC problem) from the solution of the SDP relaxation. By far, the most popular heuristic solving the unconstrained MSSC problem is the $k$-means algorithm \cite{lloyd1982least}. Given the initial cluster centers, $k$-means proceeds by alternating two steps until convergence. In the first step, each observation is assigned to the closest cluster center, whereas in the second step, each center is updated by taking the mean of all the data points assigned to it. The algorithm terminates when the centers no longer change. 
Like other local search heuristics, the quality of the clustering found by $k$-means highly depends on the choice of the initial cluster centers \cite{pena1999empirical, franti2019much}. 

To generate high-quality solutions, we propose a two-phase rounding procedure. In the first phase, we extract the initial cluster centers from the solution of the SDP relaxation solved at each node. Therefore, we exploit the quality of the SDP solution by solving a sequence of relaxations where the underlying clusters become more clearly defined in each cutting-plane iteration. In the second phase, by using integer programming tools, we design a local search procedure inspired by $k$-means that improves the initial clustering while satisfying the cardinality constraints. We run this procedure at each node and we incorporate pairwise constraints described in Section \ref{sec:branching} whenever branching decisions are taken.

Consider the ccMSSC formulation and recall that $\pi_j = X_{:,j}$ is the indicator vector of cluster $j$.
In general, the solution $\tilde{X}$ of problem \eqref{prob:mlSDP} represents a soft assignment matrix since its elements are between 0 and 1. To build a soft assignment from a solution $\{(\tilde{\Pi}_j, \tilde{\pi}_j)\}_{j=1}^k$ of problem \eqref{prob:vlSDP}, we stack the cluster indicator variables as columns of the matrix $\tilde{X} = [\tilde{\pi}_1, \dots, \tilde{\pi}_k]$. 
Recall that $\mathcal{F}$ is the set of all assignments with known cluster sizes as defined in equation \eqref{set:ass}. Furthermore, denote by $\mathcal{ML}$ and $\mathcal{CL}$ the sets of assignment variables satisfying must-link and cannot-link constraints, respectively.
In order to obtain a feasibile clustering, we find the closest assignment matrix $\bar{X}_0$ to $\tilde{X}$ with respect to the Frobenius norm by solving
{\small
\begin{equation}
\label{prob:recover_X}
\begin{split}
    & \min_{X} \left\{ \|X - \tilde{X} \|_F^2 : X \in \mathcal{F} \cap \mathcal{ML} \cap \mathcal{CL} \right\}\\ &= \max_{X} \left\{\ip{X}{\tilde{X}} : X \in \mathcal{F} \cap \mathcal{ML} \cap \mathcal{CL} \right\}.
\end{split}
\end{equation}
}
At the root node, that is when $[l_{=}] = [l_{\neq}] = \emptyset$, problem \eqref{prob:recover_X} can be solved in polynomial time because its constraint matrix is totally unimodular, implying that its LP relaxation is exact.
When pairwise constraints are added at each node, the unimodularity does not hold anymore. However, in practice the problem can be quickly solved to global optimality by using off-the-shelf integer programming solvers.  
The solution $\bar{X}^0$ is a feasible clustering and thus the corresponding centroids $\bar{m}_j = \frac{1}{c_j} \sum_{i=1}^{n} \bar{X}^0_{ij} p_i$ for all $j \in [k]$ can be used to initialize the local search procedure illustrated in Algorithm \ref{alg:cckmeans}.

Note that in principle, if the solver would fail in solving problem \eqref{prob:recover_X}, we could use the same initialization found at the root node or any other random initialization. However, in all our experiments, off-the-shelf software always solved problem \eqref{prob:recover_X} in negligible time.
\begin{algorithm}
\caption{Find a feasible clustering via the SDP solution}
\label{alg:cckmeans}
\begin{algorithmic}
\State{\textbf{Input}: Data points $p_1, \dots, p_n$, cluster sizes $c_1, \dots, c_k$, initial centers $\bar{m}_1, \dots, \bar{m}_k$ extracted from the SDP solution, branching decisions as must-link $[l_{=}]$ and cannot-link $[l_{\neq}]$ constraints}.
\While{$\textit{there are changes in}$ $\bar{m}_1, \dots, \bar{m}_k$}
\State{$\bar{X} = \argmin_{X} \left\{\sum_{i=1}^{n} \sum_{j=1}^{k} X_{ij} \| p_i - \bar{m}_j \|_2^2 \ : \ X \in \mathcal{F} \cap \mathcal{ML} \cap \mathcal{CL} \right\}$}.
\State{$\bar{m}_j = \argmin_{m_j} \left\{\sum_{i=1}^{n} \sum_{j=1}^{k} \bar{X}_{ij} \| p_i - m_j \|_2^2 \right\} = \frac{1}{c_j} \sum_{i=1}^n \bar{X}_{ij} p_i \quad \forall j$}.
\EndWhile
\State{\textbf{Output}:} $\bar{X}$
\end{algorithmic}
\end{algorithm}

Steps 1 and 2 of Algorithm \ref{alg:cckmeans} are reminiscent of a single iteration of $k$-means algorithm for unconstrained clustering. Specifically, Step 1 computes the optimal assignment of each point to the nearest cluster center while adhering to both cardinality and pairwise constraints. The feasible set of the ILP is the same of problem \eqref{prob:recover_X}, and the same considerations hold on the total unimodularity property at the root node, and on the practical behaviour of off-the-shelf solvers.
The minimization in Step 2, instead, admits a closed form solution, and like $k$-means, is given by the sample average of data points assigned to each cluster. Similarly to the $k$-means algorithm, there is no guarantee that the iterates generated from our heuristic will converge to the global minimizer. Differently from the heuristic proposed in \cite{rujeerapaiboon2019size}, our algorithm finds the initial cluster centers by exploiting SDP relaxations strengthened through cutting planes and improves the quality of the initial clustering by appending some iterations of the $k$-means algorithm with both cardinality and pairwise constraints. Numerical experiments show that proposed heuristic is able to find the optimal solution at the root or in the first few nodes.

The overall branch-and-cut algorithm is shown in Algorithm \ref{alg:bbpseudocode}. Algorithm \ref{alg:bbpseudocode} converges to a global minimum with precision $\varepsilon$, thanks to the binary branching that performs an implicit enumeration of all the possible cluster assignments. In practice, a small number of nodes is required thanks to the strength of the lower bounding routine and to the effectiveness of the heuristic, as we will show in the next section. 

\begin{algorithm}
\caption{Branch-and-cut algorithm for ccMSSC}
\label{alg:bbpseudocode}

\textbf{Input}: Gram matrix $W$, number of clusters $k$, matrix of cardinalities $C$, optimality tolerance $\varepsilon$.

\begin{enumerate}[label*=\arabic*., nolistsep] 
    \item Let $P_0$ be the initial ccMMSC problem and set $\mathcal{Q} = \{P_0\}$. 
    \item Set $X^\star = \textrm{null}$ with objective function value $v^\star = \infty$.
    \item While $\mathcal{Q}$ is not empty:
    \begin{enumerate}[label*=\arabic*., rightmargin=15pt, nolistsep]
        \item Select and remove problem $P$ from $\mathcal{Q}$.
        \item Compute a lower bound (either by VL-SDP or by ML-SDP) for problem $P$.
        \item Apply the post-processing via error bounds  and the LP-based post-processing and select the highest valid lower bound $LB$.
        \item If $v^\star<\infty$ and  $(v^\star - LB)/v^\star \leq \varepsilon$, go to \textit{Step 3}.
        \item Search for violated triangle inequalities. If any are found, add them to the current SDP relaxation and go to \textit{Step 3.2}.
        \item Extract the initial cluster centers from the solution of the SDP relaxation and run the heuristic in Algorithm \ref{alg:cckmeans} to get an assignment $X$ and an upper bound $UB$. If $UB < v^\star$ then set $v^\star \leftarrow UB$, $X^\star \leftarrow X$.
        \item Select the branching pair $(i,j)$ and partition problem $P$ into must-link and cannot-link subproblems. For each problem update $T^s$, $e^s$ and $[l_{\neq}^s]$ accordingly, add them to $\mathcal{Q}$ and go to \textit{Step 3}.
    \end{enumerate}
\end{enumerate}
\textbf{Output}: Optimal assignment matrix $X^\star$ with objective value $v^\star$

\end{algorithm}

\section{Computational Results}
\label{sec:results}
In this section, we describe the implementation details and we show numerical results on real-world instances.

\subsection{Implementation details}
Our B\&B algorithm is implemented in C++ with some routines written in MATLAB. The SDP relaxations are solved with SDPNAL+, a MATLAB software that implements an augmented Lagrangian method to solve semidefinite programming problems with bound constraints \cite{sun2020sdpnal+}. We set the accuracy tolerance of SDPNAL+ to $10^{-4}$ in the relative KKT residual and we post-process the output of the solver by using the tecniques described in Section \ref{sec:sdp_safe}. We use Gurobi \cite{gurobi} for the LP-based post-processing and for solving the integer problems in Algorithm \ref{alg:cckmeans}. We run the experiments on a laptop with Intel(R) i7-12700H CPU @ 3.50GHz having 14 cores, 16 GB of RAM and Ubuntu operating system. To improve the efficiency of the B\&B search we concurrently explore multiple nodes of the tree. 
As for the cutting-plane procedure, at each iteration, we separate at most 100000 triangle inequalities, we sort them in decreasing order with respect to the violation and we add the first 10\%. In our numerical tests, the tolerance for checking the violation is set to $10^{-4}$. Furthermore, we use the same tolerance for removing inactive inequalities. We stop the cutting-plane procedure when there are no violated inequalities or the relative difference between consecutive lower bounds is less than $10^{-4}$. Finally, we explore the B\&B tree with the best-first search strategy. The source code is available at \url{https://github.com/antoniosudoso/cc-sos-sdp}.


\begin{table}[!ht]
\centering
{\footnotesize
\begin{tabular}{|l|l|c|c|l|l|l|}
\hline
ID 	&	 Dataset   	&	  $n$   	&	  $d$ 	&	 $k$ 	&	 $c_1, \dots, c_k$ 	&	$f_\textrm{OPT}$	\\\hline
01	&	 Ruspini 	&	75	&	2	&	4	&	 15, 20, 17, 23	&	1.288e+04	\\
02	&	  BreastTissue 	&	106	&	9	&	6	&	 18. 17, 17, 18, 18, 18	&	2.371e+10	\\
03	&	  Hierarchical 	&	118	&	1798	&	4	&	 42, 45, 21, 10	&	 7.424e+06	\\
04	&	  Iris 	&	150	&	4	&	3	&	 50, 50, 50 	&	8.127e+01	\\
05	&	  HapticsSmall 	&	155	&	1092	&	5	&	 18, 34, 34, 36, 33	&	1.794e+04	\\
06	&	  UrbanLand 	&	168	&	147	&	9	&	 23, 29, 14, 15, 17, 25, 16, 14, 15	&	3.498e+09	\\
07	&	  Wine 	&	178	&	13	&	3	&	 59, 71, 48	&	2.398e+06	\\
08	&	  Parkinson 	&	195	&	22	&	2	&	 147, 48	&	1.364e+06	\\
09	&	  Connectionist 	&	208	&	60	&	2	&	 111, 97	&	2.805e+02	\\
10	&	  Seeds 	&	210	&	7	&	3	&	 70, 70, 70 	&	6.056e+02	\\
11	&	  Plane 	&	210	&	144	&	7	&	 30, 30, 30, 30, 30, 30, 30	&	1.693e+03	\\
12	&	  InsectEPG 	&	249	&	601	&	3	&	 89, 118, 42	&	1.360e+03	\\
13	&	  VertebralCol. 	&	310	&	6	&	3	&	 100, 60, 150	&	3.471e+05	\\
14	&	  Fish 	&	350	&	463	&	7	&	 50, 50, 50, 50, 50, 50, 50	&	4.152e+03	\\
15	&	  PowerCons. 	&	360	&	144	&	2	&	 180, 180	&	3.663e+04	\\
16	&	  MusicEmotion 	&	400	&	50	&	4	&	 100, 100, 100, 100	&	5.229e+08	\\
17	&	  GunPointAge 	&	451	&	150	&	2	&	 228, 223	&	2.110e+09	\\
18	&	  Computers 	&	500	&	720	&	2	&	 250, 250	&	3.077e+05	\\
19	&	  {EthanolLevel} 	&	500	&	1751	&	4	&	 126, 124, 126, 124	&	6.223e+03	\\
20	&	  SyntheticCon. 	&	600	&	60	&	6	&	 100, 100, 100, 100, 100, 100	&	1.771e+04	\\
21	&	  AbnormalHea. 	&	606	&	3053	&	5	&	 40, 40, 46, 129, 351	&	2.853e+04	\\
22	&	  ElectricDev. 	&	624	&	256	&	2	&	 543, 81	&	2.916e+13	\\
23	&	  ScreenType 	&	750	&	720	&	3	&	 250, 250, 250	&	4.090e+05	\\
24	&	 Gene 	&	801	&	20531	&	5	&	 300, 78, 146, 141, 136	&	1.781e+07	\\
25	&	  ECGFiveDays 	&	884	&	136	&	2	&	 442, 442	&	3.528e+04	\\
26	&	  UWaweGest. 	&	896	&	945	&	8	&	 122, 108, 106, 110, 127, 111, 112, 100	&	4.505e+05	\\
27	&	  Raisin 	&	900	&	7	&	2	&	 450, 450	&	1.293e+12	\\
28	&	  CBF 	&	930	&	128	&	3	&	 310, 310, 310	&	5.209e+04	\\
29	&	  TwoPatterns 	&	1000	&	128	&	4	&	 271, 237, 250, 242	&	1.033e+05	\\\hline
\end{tabular}
\caption{Real-world instances.  \label{tab:instances}}}
\end{table}

\begin{table}[!ht]
\centering
\footnotesize
\begin{tabular}{|l||ccc|cc||ccc|cc|}
\hline
 &\multicolumn{5}{c||}{B\&B VL-SDP} & \multicolumn{5}{c|}{B\&B ML-SDP}\\\hline
ID & Gap$_{0}$& CP & Gap$_r$ &  Nodes & Time & Gap$_{0}$& CP & Gap$_r$ &  Nodes & Time \\
\hline
01	&	0.00	&	  0 	&	0.00	&	1	&	3	&	0.00	&	0	&	 0.00  	&	1	&	3	\\
02	&	0.04	&	1	&	0.00	&	1	&	10	&	0.09	&	1	&	 0.00 	&	1	&	  8	\\
03	&	0.42	&	1	&	0.00	&	1	&	98	&	8.72	&	5	&	 0.29 	&	19	&	465	\\
04	&	0.00	&	  0 	&	0.00	&	1	&	5	&	0.03	&	1	&	 0.00  	&	1	&	 9 	\\
05	&	1.55	&	  6 	&	0.03	&	  3 	&	  1276 	&	2.14	&	4	&	  0.20	&	31	&	 1662 	\\
06	&	3.02	&	7	&	1.29	&	49	&	16422	&	5.69	&	3	&	2.21	&	165	&	11471	\\
07	&	0.00	&	0	&	0.00	&	1	&	41	&	4.38	&	2	&	 0.75 	&	7	&	291	\\
08	&	0.03	&	  1 	&	0.00	&	1	&	61	&	1.99	&	5	&	  0.00	&	1	&	  274	\\
09	&	0.41	&	1	&	0.00	&	1	&	245	&	6.83	&	3	&	 0.05 	&	3	&	650	\\
10	&	 0.21  	&	1	&	0.00	&	1	&	54	&	0.61	&	1	&	 0.00 	&	1	&	47	\\
11	&	0.65	&	2	&	0.45	&	13	&	3834	&	0.79	&	2	&	0.64	&	87	&	1006	\\
12	&	0.04	&	1	&	0.00	&	1	&	21	&	0.53	&	2	&	0.00	&	1	&	52	\\
13	&	0.10	&	1	&	0.00	&	1	&	583	&	1.01	&	4	&	0.04	&	3	&	471	\\
14	&	0.65	&	6	&	0.02	&	5	&	7885	&	1.31	&	5	&	0.07	&	45	&	 1837 	\\
15	&	0.02	&	1	&	0.00	&	1	&	466	&	0.32	&	2	&	0.00	&	1	&	203	\\
16	&	0.69	&	5	&	0.04	&	7	&	8374	&	1.15	&	3	&	0.06	&	23	&	2297	\\
17	&	0.00	&	0	&	0.00	&	1	&	212	&	0.08	&	2	&	0.00	&	1	&	 460 	\\
18	&	0.59	&	12	&	0.00	&	1	&	2654	&	  0.87 	&	10	&	0.00	&	1	&	1086	\\
19	&	3.29	&	17	&	0.41	&	 1 (0.41\%) 	&	 - 	&	6.33	&	14	&	0.39	&	75	&	21988	\\
20	&	0.92	&	18	&	0.09	&	1	&	28511	&	0.96	&	7	&	0.05	&	3	&	727	\\
21	&	0.85	&	4	&	0.65	&	 9 (0.54\%) 	&	 - 	&	0.88	&	5	&	0.74	&	89	&	27013	\\
22	&	0.00	&	0	&	0.00	&	1	&	432	&	1.27	&	3	&	0.04	&	1	&	584	\\
23	&	0.75	&	14	&	0.66	&	 3 (0.65\%) 	&	 - 	&	0.77	&	15	&	0.08	&	5	&	8643	\\
24	&	0.00	&	0	&	0.00	&	1	&	7109	&	1.16	&	2	&	0.01	&	3	&	 2046 	\\
25	&	0.04	&	0	&	 0.04  	&	1	&	348	&	0.84	&	2	&	 0.05 	&	1	&	431	\\
26	&	2.39	&	6	&	1.53	&	 1 (1.53\%) 	&	 - 	&	2.85	&	16	&	0.72	&	35	&	38171	\\
27	&	0.01	&	 0  	&	0.01	&	1	&	1027	&	0.70	&	1	&	 0.01 	&	1	&	402	\\
28	&	0.69	&	10	&	0.51	&	 5 (0.49\%) 	&	 - 	&	0.91	&	9	&	0.09	&	1	&	4352	\\
29	&	1.77	&	15	&	1.12	&	 1 (1.12\%) 	&	 - 	&	2.67	&	19	&	0.32	&	27	&	34396	\\\hline
\end{tabular}
    \caption{The  ``–'' signs indicate that
the problem instance could not be solved within a time limit of 12 hours (43200 seconds). The times are in seconds and all the reported gaps are in percentage.}
    \label{tab:complete}
\end{table}

\subsection{Results on real-world instances}
In Table \ref{tab:instances}, we report 29 real-world datasets for classification problems with number of data points $n \in [75, 1000]$, number of features $d \in [2, 20531]$, number of clusters $k \in [2, 9]$, and $f_\textrm{OPT}$ that is the optimal value certified by our methodology. They can all be downloaded at UCI\footnote{\url{http://archive.ics.uci.edu/ml}}
and UCR\footnote{\url{http://www.cs.ucr.edu/~eamonn/time_series_data_2018}} websites. Following the related literature \cite{rujeerapaiboon2019size, haouas2020exact}, the number of clusters is assumed to be equal to the number of classes. Furthermore, we set the cluster cardinalities $c_1, \dots, c_k$ to the numbers of true class occurrences in each dataset. We implement two versions of the B\&B algorithm where the only difference is the adopted SDP relaxation: in B\&B VL-SDP we use problem \eqref{prob:vlSDP}, whereas in B\&B ML-SDP we use problem \eqref{prob:mlSDP}. We add triangle inequalities in a cutting-plane fashion as described in Section \ref{sec:cutting_plane}. 
We require the optimality tolerance on the percentage gap of $\varepsilon = 0.01\%$ for instances with $n < 500$ and $\varepsilon = 0.1\%$ for instances with $n \geq 500$, i.e., we terminate each method when $100\cdot\frac{UB - LB}{UB} \leq \varepsilon$, where UB and LB denote the best upper and lower bounds, respectively. In all our experiments, we set a time limit of 12 hours of computing time.

In Table \ref{tab:complete}, we compare the two versions of our B\&B algorithm. Here, we report the instance id and for each method some statistics relative to the root node, the total number of processed nodes and the computational time in seconds. As for the root node, we report the percentage gap after solving the SDP without adding triangle inequalities (Gap$_0$), the number of cutting-plane iterations (CP), and the percentage gap at the end of the cutting-plane procedure (Gap$_r$).
Whenever the time limit is reached, we set the time to ``-'' and we report in brackets in the ``Nodes'' column the gap when the algorithm stops. It turns out that B\&B ML-SDP solves to the required precision all the instances within the time limit. One the other hand, B\&B VL-SDP outperforms B\&B ML-SDP when $n$ and $k$ are small. We also stress that when B\&B VL-SDP stops for the time limit, the optimality gap is always smaller than $2\%$. Looking at the table, the results confirm that the bound produced by the vector lifting relaxation is stronger (the number of nodes is always smaller than the one of B\&B ML-SDP) and computationally tractable for small $n$ and $k$. When $n$ and $k$ increase, the efficiency of the matrix lifting relaxation allows to process a larger number of nodes in a smaller amount of time. Note that the reduction in size of the SDP described in Section \ref{sec:branching} helps to limit the computational time when the number of nodes increases. Therefore, the influence of the size reduction is stronger for B\&B ML-SDP where the number of nodes is larger and hence the branching decisions are more frequent.

Many instances are solved at the root node (especially with B\&B VL-SDP), and this depends both on the strength of the lower bound (thanks to the cutting-plane procedure) and on the excellent upper bound produced by the heuristic. The statistics on the root node confirm that the vector lifting relaxation is stronger: in 8 of 29 instances it is tight (no cutting-plane iterations are performed), and in general the number of cutting-plane iterations is much smaller than the one needed for the matrix lifting relaxation. However, at the end of the root node, the difference in gap between the two relaxations becomes almost negligible, thanks to the cutting-plane procedure. Results confirm that adding inequalities significantly reduces the gap, and allows in many cases to close the gap at the root node: with B\&B VL-SDP it happens on 19 out of 29 instances, and with B\&B ML-SDP on 13 instances out of 29. In general, the gap at the end of the root node is always below 2\% for B\&B VL-SDP and 3\% for B\&B ML-SDP. When $k$ increases, the gap increases, and also finding the global minimum by the heuristic becomes harder. Indeed, the only instances where the global minimum is not found at the root node are: UrbanLand, Plane, EthanolLevel, AbnormalHeart, UWaweGesture. 
A significant example is UrbanLand, where $k=9$. This is the instance where the gap at root node is higher, and also the optimal solution is not found at the root, but it is found after 16 and 85 nodes in B\&B VL-SDP and B\&B ML-SDP, respectively. 
{To get a better understanding of the cutting-plane contribution at the root node, in Figures \ref{fig:cp1} and \ref{fig:cp2} we plot for VL-B\&B and ML-B\&B the relative gap versus the number of CP iterations whenever $\textrm{CP} \geq 3$. The ``x'' marker indicates that the global upper bound has been updated at the corresponding CP iteration.
These plots reveal two insights. First, during the initial iterations, the gap diminishes at a faster rate. However, it's noteworthy that the subsequent iterations are still valuable, as they lead to further gap reduction and frequently result in UB updates due to the improved SDP solution.}


Although the method in \cite{rujeerapaiboon2019size} does not guarantee globally optimal solutions, the SDP bound and the rounding heuristic are able to prove the optimality of Iris, Parkinson and Seeds: problem \eqref{prob:vlSDP} is tight for Iris and one cutting-plane iteration is sufficient for solving Parkinsons and Seeds. Despite the relevance of ccMSSC problem the constraint programming solver in \cite{haouas2020exact} is the only computational study on exact approaches appeared in the literature, and the code is available. Thus, we run their code on our machine considering all the small-scale instances, i.e., $n\le 210$, and compare the performance of our B\&B algorithms against it. The algorithm in \cite{haouas2020exact} is able to find the globally optimal solution of BreastTissue, Iris, Parkinsons, Ruspini and Wine in 14, 9, 29397, 8, and 8224 seconds, respectively.  Hence the method is competitive with our approaches apart from Wine and Parkinsons, where we show vastly improved results in terms of computational time. For the remaining small-scale instances, we report the relative gap after the time limit of 12 hours: Hierarchical 27.32\%, HapticsSmall 56.01\%, UrbanLand 34.73\%, Connectionist 19.03\%, Seeds 1.49\%. Given the large gap on these small-scale instances, we did not run their method on the large-scale ones. 
 
Summarizing, we are able to solve for the first time instances having sizes approximately 10 times larger than those solved by previous exact approaches. The cutting plane procedure is fundamental for closing the gap, combined with the effectiveness of the heuristic, that allows to find the global minimum at the root node or in a few nodes. When the size is small, the vector lifting relaxation is preferable thanks to the strength of the bound, but when $n$ and $k$ increase, the matrix lifting relaxation allows to get a more tractable node in terms of computational time, and the cutting plane procedure makes the bound still competitive with respect to the vector lifting.

\begin{figure}[!ht]
    \centering
    \includegraphics[scale=0.52]{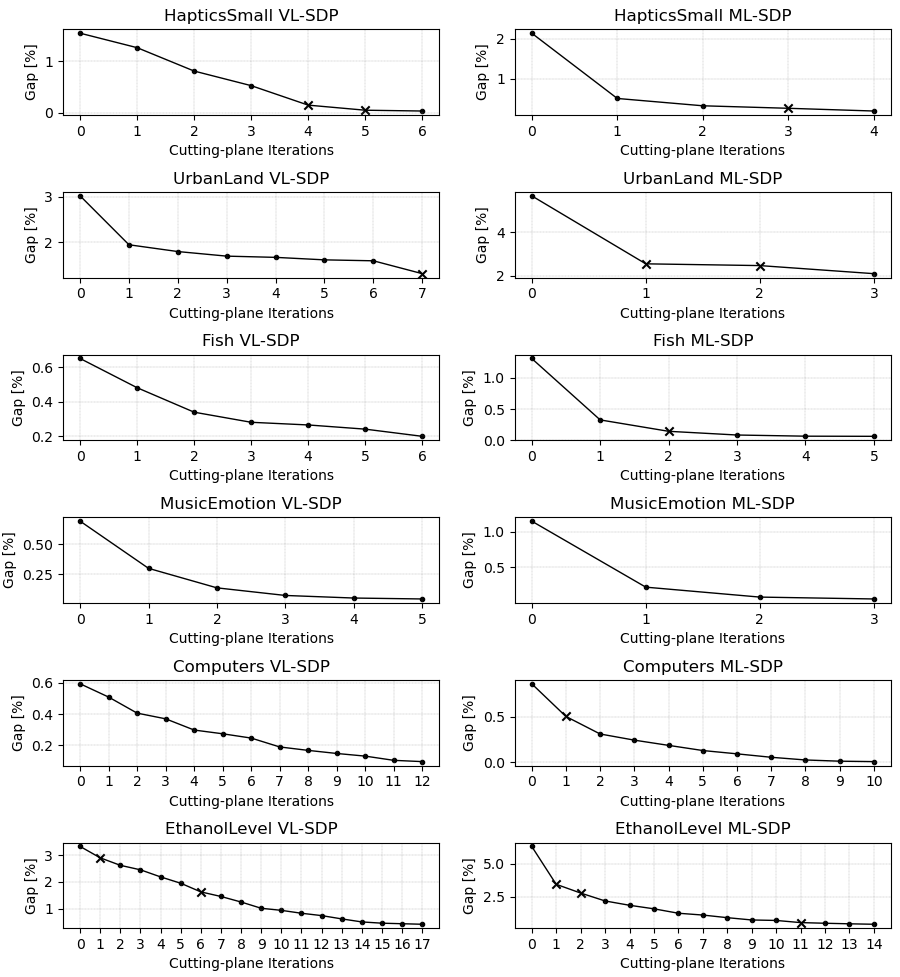}
    \caption{{Gap versus cutting-plane iterations at the root node. The ``x'' marker indicates that the global upper bound has been updated at the corresponding iteration.}}
    \label{fig:cp1}
\end{figure}

\begin{figure}[!ht]
    \centering
    \includegraphics[scale=0.52]{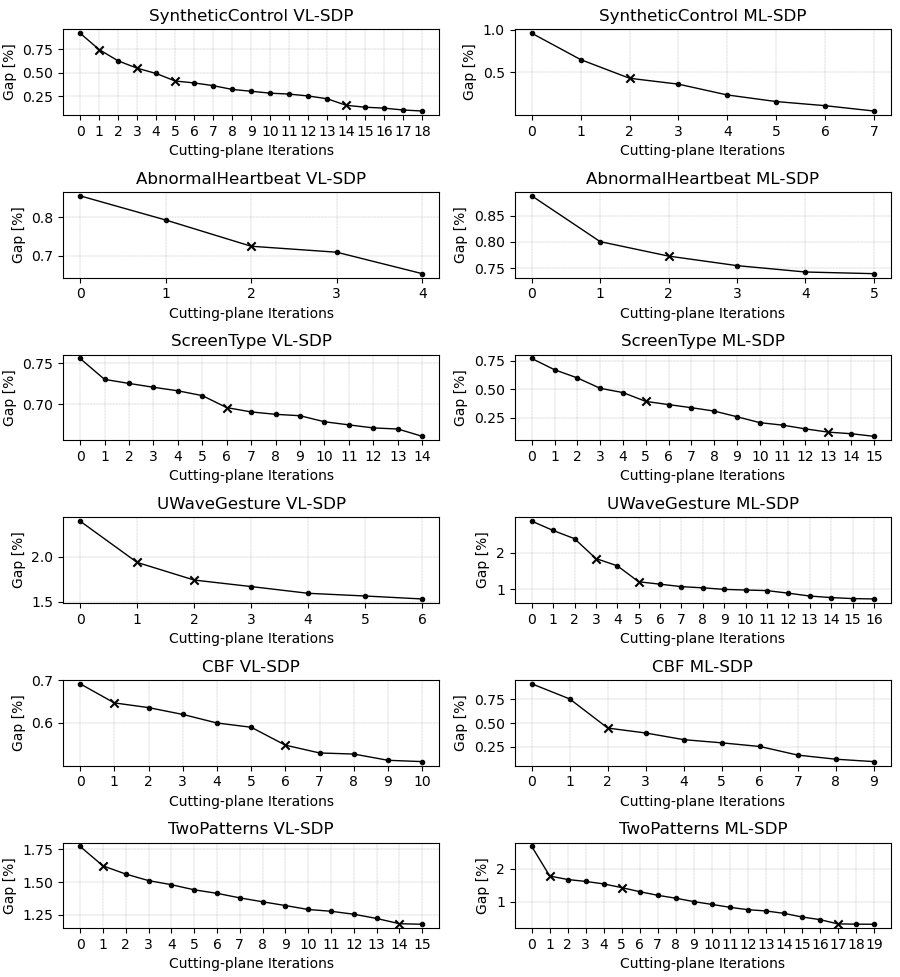}
    \caption{{Gap versus cutting-plane iterations at the root node. The ``x'' marker indicates that the global upper bound has been updated at the corresponding iteration.}}
    \label{fig:cp2}
\end{figure}

\section{Conclusions}
\label{sec:conclusions}
In this paper, we proposed an SDP-based branch-and-cut algorithm for solving MSSC with strict cardinality constraints. For computing the lower bound, we used the SDP relaxation recently proposed in \cite{rujeerapaiboon2019size} for small-scale instances, whereas for large-scale ones we derived a new SDP relaxation. We implemented a cutting-plane algorithm to strengthen the bounds provided by both relaxation by adding polyhedral cuts. For the upper bound computation, we designed a constrained variant of $k$-means algorithm and we initialized it with the solution of the SDP relaxation solved at each node. 
Numerical results impressively exhibit the efficiency of our solver: when using the new SDP bound, we can solve real-world instances up to 1000 data points, whereas when using the bound in \cite{rujeerapaiboon2019size} we can only solve small-size problems and guarantee an optimality gap smaller than 2\% on larger instances. To the best of our knowledge, no other exact solution methods can handle real-world instances of that size. 
{An interesting research direction to improve our methodology is to study facial reduction for both the VL-SDP and ML-SDP relaxations, following \cite{li2021strictly}. This should make the solution of the SDPs faster and more robust, leading to a more efficient branch-and-cut algorithm.}
Furthermore, an immediate future research direction is to design global algorithms for clustering problems with fairness constraints. As introduced in \cite{chierichetti2017fair}, the goal of fair clustering is to find a partition where all the clusters are balanced with respect to some protected attributes such as gender or religion.

\section*{Acknowledgements}
Veronica Piccialli has been supported by PNRR MUR project PE0000013-FAIR.

\section*{Competing Interests}
The authors have no relevant financial or non-financial interests to disclose.

\bibliography{abbr, references}


\end{document}